%% file: ZUoperatorsII.tex
\documentclass[]{compositio}
\usepackage[bottom]{footmisc}
\usepackage{xr}
\usepackage[dvipsnames,table,xcdraw]{xcolor}

\usepackage[colorlinks=true]{hyperref}
	\hypersetup{urlcolor=RubineRed,linkcolor=RoyalBlue,citecolor=ForestGreen}

\usepackage{cooltooltips}
\usepackage{graphicx}

\usepackage{forest,adjustbox}
\usetikzlibrary{shapes.geometric,decorations.markings}
\usetikzlibrary{positioning,arrows.meta,decorations.pathreplacing}

\forestset{%
  my tree/.style={
    for tree={
      circle,
      draw,
      inner sep=1pt,
      l sep'=5mm,
      l'=5mm,
      s sep'=2mm,
    },
  }
}
\usepackage{caption}

	\newtheorem{theorem}{Theorem}[subsection]
	\newtheorem{lemma}[theorem]{Lemma}
	\newtheorem{proposition}[theorem]{Proposition}
	\newtheorem{corollary}[theorem]{Corollary}
	\newtheorem*{theorem*}{Theorem}
	\newtheorem*{lemma*}{Lemma}
	\newtheorem*{proposition*}{Proposition}
	\newtheorem*{corollary*}{Corollary}
		\theoremstyle{definition}
		
		\newtheorem{definition}[theorem]{Definition}
		\newtheorem{example}[theorem]{Example}
\newtheorem{remark}[theorem]{Remark}

\input{settings_custom.tex}

\newtheorem{innercustomtheorem}{Theorem}
\newenvironment{customtheorem}[1]
{\renewcommand\theinnercustomtheorem{#1}\innercustomtheorem}
  {\endinnercustomtheorem}

\usepackage[scr=boondoxo,scrscaled=1.05]{mathalfa}%forcaligraphiclowercase

\usepackage[intoc,refpage]{nomencl}
\def\@@@nomenclature[#1]#2#3{%
\def\@tempa{#2}\def\@tempb{#3}%
\protected@write\@glossaryfile{}%
{\string\glossaryentry{#1\nom@verb\@tempa @[{\nom@verb\@tempa}]% 

nompageref{\begingroup\nom@verb\@tempb\protect\nomeqref{\theequation}}}% 
{\thepage}}%
 \endgroup
 \@esphack}

\usepackage{ifthen}
\renewcommand{\nomgroup}[1]{%
\ifthenelse{\equal{#1}{A}}{\item[\textbf{Chapter II \S 1 \& \S 2}]}{%
\ifthenelse{\equal{#1}{B}}{\item[\textbf{Chapter II \S 3}]}{%
\ifthenelse{\equal{#1}{C}}{\item[\textbf{Chapter III}]}{
\ifthenelse{\equal{#1}{D}}{\item[\textbf{Chapter IV}]}{
\ifthenelse{\equal{#1}{E}}{\item[\textbf{Chapter V}]}{
\ifthenelse{\equal{#1}{F}}{\item[\textbf{Chapter VI}]}{
\ifthenelse{\equal{#1}{G}}{\item[\textbf{Chapter VII}]}{
\ifthenelse{\equal{#1}{H}}{\item[\textbf{Chapter V}]}{
}}}}}}}}}
\makenomenclature

\usepackage{aligned-overset}

\begin{document}

\title{\textsc{The ring of $\mathbb{U}$-operators: Definitions and Integrality}}
%    Remove any unused author tags.

%    author one information
\author{Reda Boumasmoud}
\email{reda.boumasmoud@imj-prg.fr \& reda.boumasmoud@gmail.com}
\address{Institut de Mathématiques de Jussieu-Paris Rive Gauche (IMJ-PRG)\\
Sorbonne Université and Université de Paris, CNRS, F-75006 Paris, France.}
%\thanks{The author was supported by the Swiss National Science Foundation grant \#P2ELP2-191672.}

% Enter details of editor communicating this article
%\communicated{Brian Conrad}

%\thanks{The author was supported by the Swiss National Science Foundation grant \#PP00P2-144658 and \#P2ELP2-191672.}

%%Les quelques lignes qui suivent, sont la pour modifier subject class rendu, pour visulaiser 2020 .. comme il n y a pas eu encore de mise a jour!
\makeatletter 
\@namedef{subjclassname@2020}{%
\textup{2020} Mathematics Subject Classification}
\makeatother


%\keywords{}
%\date{}

\begin{abstract}
In this paper, we define and study the arithmetic of the ring of $\mathbb{U}$-operators for reductive $p$-adic groups. 
These operators generalize the notion of "successor" operators for trees with a marked end. 
We show that they are integral over the spherical Hecke algebra. This integrality intervenes crucially in the construction of Euler systems obtained from special cycles of general Shimura varieties and in the generalization of the famous Eichler--Shimura relation.
\end{abstract}

%%  LaTeX will not make the title for the paper unless told to do so.
%%  This is done by uncommenting the following.
%%

 \maketitle

%%
%% LaTeX can automatically make a table of contents.  This is done by
%% uncommenting the following:
%%

\tableofcontents

\section{Introduction}
The Bruhat--Tits building encodes in a geometrico-combinatorial way the structure of connected reductive groups over a (nonarchimedean) local field $F$. 
The theory of these buildings is a unifying principle with a striking range of applications, for example 
in the smooth representation theory of the connected reductive groups 
which is at the heart of the local Langlands program whose objective is the understanding of the absolute Galois group of $F$, 
in the classification of algebraic groups, 
in Kac--Moody groups which used in theoretical physics, 
in combinatorial geometry with application in computer science, 
and in the study of rigidity phenomena in spaces with negative curvature.

The main purpose of this work is to highlight and investigate an unexplored aspect in this theory, with intriguing applications in arithmetic algebraic geometry. 
More precisely,  
we define and study a ring of $\mathbb{U}$-operators that generalize the notion of "successor" operators in introduced \cite[\S 3.2]{BBJ18} for the $\U(3)\times \U(2)$ case and denoted by $\mathcal{U}_V\mathcal{U}_W$. 
Our main goal is to establish their integrality over the \emph{spherical} Hecke algebra generalizing \cite[Lemma 3.3]{BBJ18}. 
This integrality intervene crucially in the construction of Euler systems obtained from special cycles of general Shimura varieties \cite{Unitarynormrelations2020,Vunitarynormrelations2020} and in the generalization of the famous Eichler--Shimura relations \cite{Seedrelations2020,ESboumasmoud}.

\subsection{Origin of the problem}
{The relation between Galois and Hecke actions on various objects ({{which we call a distribution relation}}) has been crucial in modern number theory. 
For instance, in the construction of Galois representations associated to automorphic forms and of Euler systems. 
The goal of this section is to show how some combinatorial operators emerge naturally in the study of distribution relations of CM points in Shimura curves and capture the interaction between the Galois action and Hecke action.}
Classically, the distribution relations of Heegner points are obtained using the description of the modular curve as the space of moduli of elliptic curves with a cyclic subgroup of order $N$, see for example \cite[Proposition 3.10]{darmon:ratpoints}. Which is a serious a limitation, as we do not always have such a moduli interpretation for general Shimura varieties we are interested in.
But, as we will see in the next section (for which we refer the reader to \cite{cornut-vatsal:durham}  for all the missing details), Deligne's adelic definition of Shimura varieties simplifies the proof of these relations. 
This alternative allows the vizualization of the two Galois / Hecke actions in a combinatorial framework as we will observe in Figure \ref{fig1}. 
As a consequence, the distribution relations above follow in an almost obvious way.

\subsubsection{Shimura curves}

Let $E/F$ be a CM extension.
We put $\A_{F,f}=\A_{\Q,f} \otimes_\Q F= \widehat{\cO}_F \otimes \Q$ for the ring of finite adeles and set $\Hom_{\Q}(F,\R)=\{\iota_1, \dots \iota_d \}$. 
Let $\Sigma$ be any finite set of finite places of $F$ such that $|\Sigma| + [F : \Q]$ is odd. %, where $d:=[F : \Q]$. 
Let $\mathbb{B}$ be the quaternion algebra over $F$ which ramifies precisely at $\Sigma \cup \{\iota_2, \dots \iota_d \}$.

Let $\G$ be the reductive group over $\Q$ whose set of points on a commutative $\Q$-algebra $R$ is given by $\G(R) = (\mathbb{B} \otimes R)^\times$. 
We have $\G_{\R}= \prod_{\iota \in \Hom_{\Q}(F,\R)}\G_{\iota}$ where $\G_{\iota}$ is the
algebraic group over $\R$ whose points on a commutative $\R$-algebra $R$
is given by $ \G_\iota(\R) =(\mathbb{B}\otimes_{F, \iota}\R)^\times$. 
Let $\cX$ be the $\G(\R)$-conjugacy class of the morphism ${h} \colon\mathds{S}= \Res_{\C/\R} \mathds{G}_{m,\C}\to \G_{\R}$ given on $\R$-points by
$$ z= x+iy\mapsto (
\begin{pmatrix}
x & y \\-y &x
\end{pmatrix}
,1,\cdots,1).$$
where we have identified $\G_1$ with $\GL_{2,\R}$ thanks to an isomorphism of $\R$-algebras $\mathbb{B}\otimes_{F,\tau_1} \R\simeq \mathbf{M}_2(\R)$. 
We have $\G(\R)$-equivariant holomorphic diffeomorphism $\cX \iso \C\setminus\R$ given by the map $ghg^{-1} \mapsto g \cdot i $ where $\G(\R)$ acts on $\C\setminus \R$ through its first component $\G_{\iota_1}(\R) \simeq \GL_2(\R)$ by 
$\begin{pmatrix} a & b \\c &d \end{pmatrix} \cdot z = \frac{a z +b}{cz+d}$. 

The pair $(\G,\cX)$ is a Shimura datum and to  any open compact subgroup $U\subset \G(\A_f)$ we have a (proper and smooth) Shimura curve $\Sh_U(\G,\cX)$ which is defined over the reflex field $E(\G,\cX)=\iota_1(F)$. 
If $F = \Q$ and $\Sigma$ is empty, we get the classical modular curves. %These are precisely the Shimura curves that are not complete – there are no cusps to be added. 
The complex points of $\Sh_U(\G,\cX)$ are given by
$$\Sh_U(\G,\cX)(\C) = \G(\Q) \backslash (\cX \times \G(\R)/U).$$

\subsubsection{CM points}

Assume that all primes $ \p \in \Sigma$ do not split in $E$ (which guarantees that $E$ splits $\mathbb{B}$, which amounts to require that $E_v$ is a field
for every finite place in $\Sigma$). 
Accordingly, there exists a $F$-algebra homomorphism $E\hra \mathbb{B}$, which induces 
an embedding $\jmath_E \colon \T_E:=\Res_{E/\Q}\mathds{G}_{m,E} \to \G$. 
We have $\T_{\R}= \prod_{\iota \in \Hom_{\Q}(F,\R)}\T_{\iota}$, where $\T_{\iota}= \Res_{E\otimes_{F,\iota} \R/\R}\mathds{G}_{m,E\otimes_{F,\iota} \R}$ and the embedding $\iota_1$ induces an isomorphism $\mathds{S} \iso \T_{\iota_1}$ which we still denote by $\iota_1$. 
Consider the homomorphism $s^\tau \colon \mathds{S} \to \T_\R $ given on $\R$-points by $z \mapsto (z^\tau, 1, \cdots ,1)$ for $\tau \in \Gal(\C/\R)$. 
We may assume that $\jmath_{E,\R}\circ s^\tau $ and $h$ belongs to the same connected component of $\cX$.

Let $\cZ_{\G,U}(E)$ denote the set of CM points with complex multiplication by $E$, i.e. $z_{x,g}=(x \times g)\in \cZ_{\G,U}(E)$ for some $g\in \G(\A_f)$ and $x\in \cX$ a homomorphisms that factors through $\T_{E,\R}$, but each of the two connected components of $\cX$ contains exactly one such a morphism. 
Accordingly, the natural map $\G(\A_f)\to \cZ_{\G,U}(E)$, induces the bijection
$$\cZ_{\G,U}(E)\simeq \T_E(\Q) \backslash \G(\A_f)/ U.$$
Write $z_g$ for the special point that is the image of $g\in \G(F)$. 
The points in $\cZ_{\G,U}(E)$ are algebraic, defined over the maximal abelian extension $E^{ab}$ of $E$. 
We refer to \cite[\S 3 \& \S 3.2]{cornut-vatsal} for a more detailed discussion of CM points on Shimura curves. 
There is an obvious left action of $\T(\A_f)$ on $\cZ_{\G,U}(E)$ given by $ t \cdot z_{g}=z_{gt}$ for $t\in \T_E(\A_f) $ and $g \in \G(\A_f)$, 
which factors through the reciprocity map $\rec_E \colon  \T_E(\A_f) \twoheadrightarrow \Gal(E^{ab}/E)$. 
Therefore, 
$\Gal(E^{ab}/E)$ acts on $\Z[\cZ_{\G,U}(E)]$ as follows; for every $\sigma \in \Gal(E^{ab}/E)$, let $h_\sigma \in \T_E(\A_f)$ be any element satisfying $\rec_E(h_\sigma)=\sigma_{|_{E^{ab}}}$ we have
$$\sigma({z}_g)=h_\sigma \cdot {z}_g={z}_{h_\sigma g}, \forall g\in \G(\A_f).$$ 
For any $g\in \G(\A_f)$, the $\T_E(\A_f)$-stabilizer of the CM point $z_g$ is $\T_E(\Q)(\T_E(\A_f)\cap gUg^{-1})$ and $z_g$ is said to be defined over the extension $E[g]\subset E^{ab}$ fixed by $\rec_E(\T(\A_f) \cap gUg^{-1})$. 
Moreover, we have a Galois equivariant left action of the Hecke algebra $\cH_U(\Z) := \End_{\Z[\G(\A_f)]}\Z[\G(\A_f)/U]$. 

Let $\p$ be a prime of $F$ with residue field $k_F$ of characteristic $p$ and order $q=p^r$. 
Assume that $\mathbb{B}$ is split at $\p$, and consider a compact open subgroup $U=U^{\p} \times U_{\p}$ of $\G(\A_f)$, where $U_p$ is a maximal order ($\simeq \GL_2(\cO_{F_\p})$)
in $\mathbb{B}_\p \simeq \mathbb{M}_2(F_{\p})$ and $U^{\p}$ is a compact open subgroup of $\widehat{\mathbb{B}}^{\times, \p} =\{g=g^{\p} \times g_{\p} \in\widehat{\mathbb{B}}^{\times} \colon g_{\p} = 1\}$. 
Let $\mathfrak{P}$ be a prime of $E$ sitting above $\p$ and $\Fr_{\mathfrak{P}}\in \Gal(E^{ab}/E)$ be the corresponding geometric Frobenius. 
Put $\epsilon_\p= -1, 0$ or $1$ depending upon whether $\p$ remains inert, ramifies or splits in $E$.

\subsubsection{(Local) Galois and Hecke actions}
We are interested in relating the action of the “decomposition group at
$\p$”, i.e. $D_\p= \Gal(E_{\p}^{ab}/E_\p)$ to the action of the local Hecke algebra $\cH_{U_\p}(\Z)=\End_{\Z[\G(F_\p)]}\Z[\G(F_\p)/U_\p]$. 
For this, consider the decomposition 
$$\Gal(E^{ab}/E)\backslash \cZ_{\G,U}(E) \simeq \widehat{E}^{\times,\p}\backslash \widehat{\mathbb{B}}^{\times, \p}  / U^\p \times E_\p^\times \backslash {\mathbb{B}}_\p^{\times} /U_\p.$$
The stabilizer of any $z \in \cZ_{\G,U}(E)$ in $E_\p^\times$ equals $\cO_{c_\p(z)}^\times:= (\cO_{F_\p} + \p^{c_\p(z)} \cO_{E_\p})^\times$ for some unique integer $c_\p(z) \in \N$. 
The map $c_\p \colon  \cZ_{\G,U}(E) \to \N$ is a Galois invariant fibration with the property; any $z\in   \cZ_{\G,U}(E)$ is fixed by the closed subgroup $\rec_E(\cO_{c_\p(z)}^\times) \subset \Gal(E^{ab}/E)$. 

Set $\cZ_{\ge 1}$ for the set of CM points $z\in \cZ_{\G,U}(E)$ such that $c_\p(z)\ge 1$. 
We are interested in  expressing
$$
\Tr(z):= \sum_{\lambda \in \cO_{c_\p(z)}^\times/\cO_{c_\p(z)-1}^\times} \rec_E(\lambda) (z), \forall z\in \cZ_{\ge 1}
$$
in terms of the (local) Hecke algebra action.
%Let $E[\p^n]/E$ be the abelian extension of $E$ so that $\Gal(E[\p^n]/E) \simeq \A_F^\times / E^\times E_\infty^\times \widehat{\cO}_{\p^n}^\times$, where $\widehat{\cO}_{\p^n}$ is the $\cO_F$-order $\cO_F + \p^n \cO_E$ of conductor $\p^n$.

\subsubsection{The tree captures the two actions}\label{heckegalsec}
Let $\cT$ be the Bruhat--Tits building of $\GL_2$ over $F_\p$. 
This is a connected tree in which every node has $q + 1$ neighbours. 
Let $\circ \in \cT$ be the hyperspecial point fixed by $U_\p$ and $\cT^\circ:= \{v \in  \mathbb{B}_\p^{\times} \cdot \circ\setminus \circ\}$. 
We have a map $ \pi_\p\colon \Z[\cT^\circ \times \widehat{\mathbb{B}}^{\p,\times}  / U^\p  ] \to \Z[\cZ_{\G,U}(E)]$ of $( {E}_\p^{\times}, \cH_{U_\p}(\Z))$-bi-modules. 
Moreover, the image of $\pi_\p$ is precisely $\Z[\cZ_{\ge 1}]$.

The two actions we are interested in can be visualised through the geometry of the tree $\cT$:
Indeed, the (local) Hecke algebra consists of adjacency operators, meaning the basic element  $T_\p:=[U_\p^\times \text{diag}(\varpi,1) U_\p^\times] \in \cH_{U_\p}(\Z)$, sends a vertex to the formal sum of its neighbours. 
For the (local) Galois action, the $E_\p^\times$-orbit of CM point $z$ is precisely all the elements $z'$ such that $z'^\p=z^\p$ and $c_\p(z)=c_\p(z')$, which corresponds to the set of vertices at distance from the "base" equals to $c_\p(z)$. 
{Here, the base means the set of vertices with conductor $0$. 
The size of this set depends on $\epsilon_\p$. 
We summarise the description of the Galois action in the following picture:}

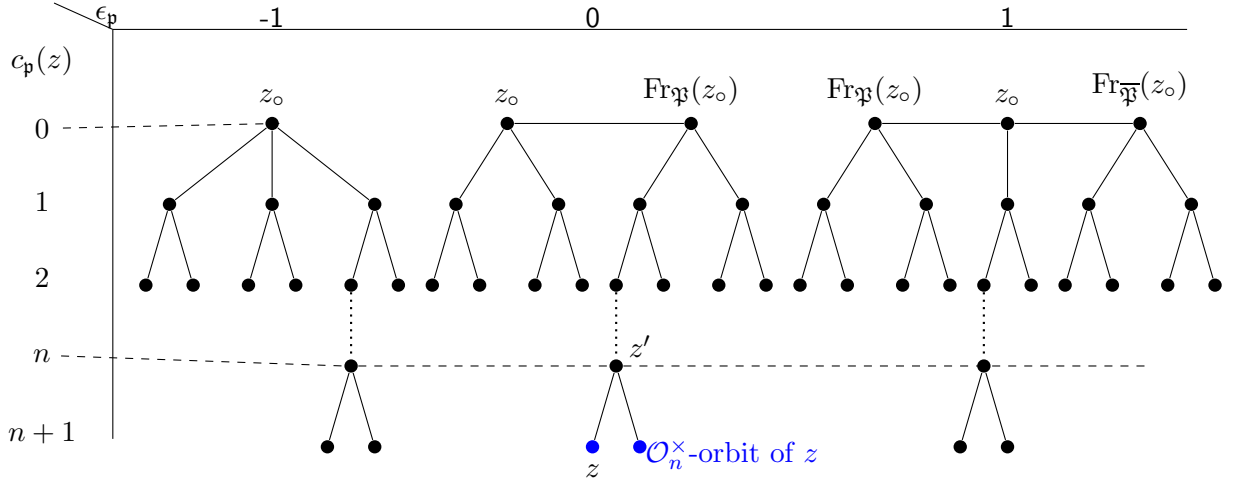
\begin{forest}
  for tree={%
    delay={%
      if n children=0{%
        label/.wrap pgfmath arg={-90:#1}{content},
        !u.s sep'+=5pt,
        !uu.s sep'+=8pt,
      }{%
        label/.option=content,
      },
      content=,
    },
    circle,
    fill,
    minimum size=5pt,
    inner sep=0pt,
  }
[, phantom,   
[$z_\circ$,name=A0,  
   [[][]]
   [[][]]
   [[[,name=X1,edge={dotted,thick}[][]]][]
]]
[$z_\circ$,name=A  
    [ [] [] ]
    [ [] [] ]
 ]
 [$\text{Fr}_{\mathfrak{P}}(z_\circ)$,name=B
  [ [[,name=X2,edge={dotted,thick}[$z$,name=S, fill=blue][,fill=blue,name=S2]
  ]] [] ]
    [ [] [] ]
 ]
[$\text{Fr}_{\mathfrak{P}}(z_\circ)$,name=C  
    [ [] [] ]
    [ [] [] ]
 ]
 [$z_\circ$,name=D
   [ [[,name=X3,edge={dotted,thick}[][]]] [] ]
 ]
  [$\text{Fr}_{\overline{\mathfrak{P}}}(z_\circ)$,name=E
    [ [] [] ]
    [ [] [] ]
]
];
\draw (C) -- (D) -- (E);
\draw (A) -- (B);
       \node (G) [left=5pt of S] {};
       \node (G2) [right=5pt of S2] {};
 \begin{scope}[font=\sffamily, >=Triangle]
      \node (Y) [above=30pt of A0] {-1};
      \node (X22) [above=1pt of X2] {}; 
      \node (X222) [right=-3pt of X22] {$z'$};
      \node (Y2) [above=30pt of D] {1};
      \node (Y1) [right=105pt of Y] {0};
      \node (Z) [left=45pt of Y] {$\epsilon_{\mathfrak{p}}$};
      \node (Z1) [below=5pt of Z] {};
       \node (Z) [left=5pt of Z1] {$c_{\mathfrak{p}}(z)$};
      \node (Z2) [below=8.5pt of Z] {$0$};
      \node (Z3) [below=13pt of Z2] {$1$};
      \node (Z4) [below=14pt of Z3] {$2$};
      \node (Z5) [below=16pt of Z4] {$n$};
      \node (Z6) [below=15pt of Z5] {$n+1$};
       \node (X4) [right=50pt of X3] {};
     \path [draw, dashed] (A0.east) -- (Z2);
      \path [draw, dashed] (Z5.east) -- (X1.west) (X1.east) -- (X2.west) (X2.east) -- (X3.west) -- (X4.east);
      \path [draw] (-222pt,5pt) -- (180pt,5pt);
      \path [draw] (-222pt,5pt) -- (-222pt,-150pt);
      \path [draw] (-222pt,5pt) -- (-244pt,15pt);
   \end{scope}
 \begin{pgfonlayer}{background}
   \node [draw=none,fill=none,   right= 16pt,minimum size=10pt] at (S.south) {\textcolor{blue}{$\cO_n^\times$-orbit of $z$}};
   \end{pgfonlayer};
\end{forest} 
\subsubsection{The successors operator}
Define the operator $u_\p \in \text{End}_{\Z[U_\p]}\Z[\cT^\circ]$ which sends a vertex $v\neq \circ $ to its successors with respect to the origin $\circ$, in other words
$$
u_\p(v)=\sum_{\text{dist}(v',\circ)=\text{dist}(v,\circ)+1}v'.
$$
Define also the predecessor operator $v_\p \in \text{End}_{\Z[U_\p]}\Z[\cT^\circ]$ sending a vertex $v\neq \circ $ to the unique $v'\in [\circ, v]$ verifying $\text{dist}(v',\circ)=\text{dist}(v,\circ)-1$. 
The operators $u_\p, v_\p$ are $\cO_{E_\p}^\times$-equivariant, hence they both extend to $\Z[\cZ_{\ge 1}]$.

\subsubsection{Combinatorial properties of $u_\p$ yields distribution relations}

Thanks to the discussion of \S \ref{heckegalsec}, we  are able to obtain the following figure which depicts our two preferred actions.
{\begin{center}
 \vspace{-1cm}
\begin{tikzpicture}[
  grow cyclic,
  level distance=2.5cm,	
  level/.style={
     level distance/.expanded={\ifnum#1>3 \tikzleveldistance/2\else\tikzleveldistance\fi},
     nodes/.expanded
  },
  level 1/.style={sibling angle=60},
  level 2/.style={dotted,sibling angle=60},
  level 3/.style={solid,sibling angle=60},
  level 4/.style={solid, sibling angle=30},
  nodes={circle,draw,inner sep=+1pt, minimum size=5pt},
  ]
\path[rotate=0]
 node (L_0)[draw,fill,minimum size=5pt]  []{}
   child  {
     node [draw,fill,minimum size=5pt][] (a){} 
       child { 
         node  [solid,draw,fill,minimum size=5pt] (b) {}
           child  {
             node  [draw,fill,minimum size=5pt] (c) {} 
               child {node [solid,fill=blue,draw,fill,minimum size=5pt] (d){}}   
               child {node [solid,draw,fill=blue,fill,minimum size=5pt] (e){}}   
               child {node [solid,draw,fill=blue,fill,minimum size=5pt] (f){}}
               child {node [solid,draw,fill=blue,fill,minimum size=5pt] (g){}}  
               child {node [solid,draw,fill=blue,fill,minimum size=5pt] (h){}} 
      }
    }
  };
                 \node [draw=none,fill=none, below =1pt] at (L_0.south east) {$z_\circ $};
                 \node [draw=none,fill=none, below right=1pt] at (a.south east) {};
                 \node (z'') [draw=none,fill=none, below right=1pt] at (b.south east) {$z''$};
                 \node (us) [draw=none,fill=none, below =1pt] at (b.south) {};
                 \node (uw) [draw=none,fill=none, left =1pt] at (b.west) {};
                 \node (ua) [draw=none,fill=none, above =1pt] at (b.north) {};
                 \node [draw=none,fill=none, below = 1pt] at (c.south east) {$z'$};
                 \node (x) [draw=none,fill=none, below = 1pt,minimum size=10pt] at (d.west) {$z$};
                 \node (xw) [draw=none,fill=none, left = 1pt,minimum size=10pt] at (d.west) {};
                 \node (xe) [draw=none,fill=none, right = 1pt,minimum size=10pt] at (d.west) {};
                 \node (yw) [draw=none,fill=none, left= 4pt] at (e.south west) {};
                 \node (ye) [draw=none,fill=none, right= 4pt] at (e.south west) {};
                 \node (zw) [draw=none,fill=none,  left= 4pt,minimum size=10pt] at (f.west) {};
                 \node (ze) [draw=none,fill=none,  right= 4pt,minimum size=10pt] at (f.west) {};
                 \node (sw)[draw=none,fill=none,  left= 4pt,minimum size=10pt] at (g.west) {};
                 \node (se) [draw=none,fill=none,  right= 4pt,minimum size=10pt] at (g.west) {};
                 \node (tw) [draw=none,fill=none, below left= 4pt,minimum size=10pt] at (h.north) {};
                 \node (tf) [draw=none,fill=none, below left= 4pt,minimum size=10pt] at (f.north) {};
                 \node (te) [draw=none,fill=none,  right= 4pt,minimum size=10pt] at (h.west) {};
                 \node (ta) [draw=none,fill=none,  above= 4pt,minimum size=10pt] at (h.west) {};
                 \node [draw=none,fill=none,   right= 6pt,minimum size=10pt] at (g.north west) {\textcolor{blue}{The support of $u_\p(z)$ is }};
                 \node [draw=none,fill=none, right=2pt, minimum size=10pt] at (ze.north)  {\textcolor{blue}{$\cO_n^\times$-orbit of $z$}};
                 \node (z''') [draw=none,fill=none,above=5pt] at (z''.north) {};
                 \node [draw=none,fill=none] at (z'''.north west) {\textcolor{red}{Hecke support of $T_\p(z')$}};
	\begin{pgfonlayer}{background}
     \draw[red,fill=red,opacity=0.1](x.south) to[closed,curve through={(xw.west) .. (yw.west) ..(zw.south) .. (sw.south) .. (tw.south) .. (us.east) .. (uw.west) ..(ua.north) .. (ta.north west) ..(te.north) .. (se.north) .. (ze.north) .. (ye.east) ..(xe.east) }](x.south);
	\end{pgfonlayer}
\end{tikzpicture}     
\end{center}
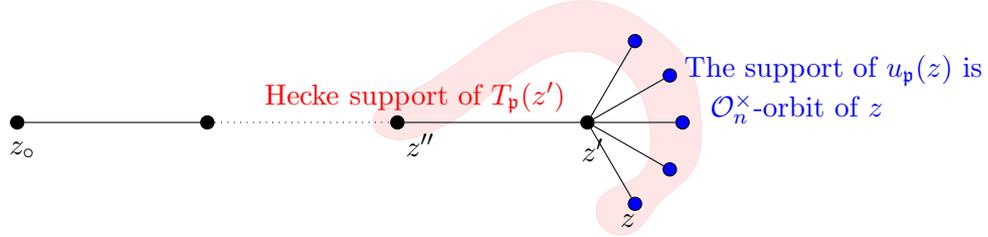
\captionof{figure}{\textbf{Hecke and Galois via $u_\p$}\label{fig1}.}}
\begin{lemma}
The operators $u_\p$ verify the following properties
\begin{enumerate} 
\item[(i)]\textbf{The Hecke side}:  
$v_\p \circ u_\p= p \, \text{Id}_{\cT^\circ} \text{($\neq u_\p \circ  v_\p $) and }T_\p=u_\p + v_\p$.
\item[(ii)]\textbf{The Galois side}: Let $z$ be a CM point with $c_\p(z)=n$. 
If $n\ge 2$, set $z'=v_\p(z)$. 
We have $c_\p(z')=n-1$ and 
$$\Tr(z)=u_\p(z')=T_\p(z')-z''.$$
\end{enumerate}
\end{lemma}
\begin{proof}
These statements can readily be extracted from Figure \ref{fig1}.
\end{proof}

{The Hecke side may be interpreted as a "lift" to characteristic zero of the well known Eichler--Shimura relation for Shimura curves $T_\p=\Fr_\p +\Fr_\p^* $ (where $\Fr_\p^*$ denotes the dual correspondence of $\Fr_\p$) \cite[Proposition 1.4.10]{SZhang2001}. 
An immediate consequence of the Hecke side is that $u_\p$ (resp. $v_\p$) is a right (resp. left) root of the Hecke polynomial $H_\p(X)=X^2
- T_\p X + p$, i.e.
$$u_\p^2 - T_\p \circ u_\p + p=0.$$
We insist on the fact that the $T_\p$, $v_\p $ and $ u_\p$ do not commute. 
This non-commutativity, was the main source of computational complications in \cite{BBJ18} for the unitary case $\G=\U(1,1)\times \U(2,1)$.
\begin{remark}
Actually, for any choice of a point $\xi \in \cT^\infty$ in the tree at infinity, one can still define a similar successor operator $u_\xi$ with respect to $\xi$. It is clear that it still a right root of the Hecke polynomial $H_\p$. 
\end{remark}

{It becomes clear that these curious combinatorial operators allow us to link and capture our two favorite actions, however they have an unpleasant limitation: they depend on the "base" which itself depends on ramification assumptions. 
For this reason, the "correct" definition of $ u_ \p $ should be the "successor" operator with respect to the point in the building at infinity corresponding to a fixed Borel subgroup. 

Although, we mainly focused on vertical distribution relations, the modified operator $u_\p$ may also be used to prove horizontal norm compatible systems of CM points, which are an essential ingredient in Kolyvagin's proof (combined with the Gross--Zagier formula) of the strongest evidence towards the Birch and Swinerton-Dyer conjecture.}

\subsection{Main results}

Let $\G$ be a reductive group defined over a non-archimedean local field $F$. 
Let $\Sbf$ be maximal $F$-split subtorus of $\G$ and $\cA$ its corresponding apartment in the reduced Bruhat--Tits building $\cB_{\red}(\G,F)$ of $\G$. 
Set $\Mbf$ and $\Nbf$ for the centralizer and the normalizer of $\Sbf$ in $\G$, respectively. We have a homomorphism of groups $\nu_N \colon \Nbf(F) \to \Aff(\cA)$. 
Let $\a\subset \cA$ be an alcove, and $a_\circ \in \overline{\a}$ a special point. 
Set $I\subset K$ for the Iwahori subgroup and the special maximal parahoric subgroup corresponding to  ${\a}$ and $a_\circ$, respectively. 
Let $\cG$ be the connected group scheme attached to $a_\circ$, and $K=\cG(\cO_F)$ be the corresponding special maximal parahoric subgroup. 
Set $I\subset K$ for the Iwahori subgroup corresponding to ${\a}$.

This is a sequel of \cite{UoperatorsI2021}, so we will borrow some notations from {\em loc. cit}.

Let $M_1 = \ker \kappa_M$ and $G_1=\ker \kappa_G$ be the kernels of the Kottwitz map for $\Mbf$ and $\G$ (see for example \cite[\S \ref{Kottwitzhom}]{UoperatorsI2021}). 
Fix any open compact subgroup $ M_1 \subset M^{\flat}\subset M$. 
By \cite[Proposition 1.2]{Lan96}, $M^{\flat}$ is contained in the maximal open compact subgroup $M^{1}$. 
Set $\Lambda_M^{\flat}:=\Mbf(F)/M^{\flat}$, $\Lambda_{\aff}^{\flat}:=(\Mbf(F) \cap G_1)^{\flat}/M^{\flat}$, $ \mathcal{R}^{\flat}:=\Z[\Lambda_M^{\flat}]$ and $ \mathcal{R}_{\aff}^{\flat}:=\Z[\Lambda_{\aff}^{\flat}]$.

For any $X\subset G$, set $X^{\flat}=M^{\flat} X$. This procedure does not alter the structure of the object we are interested in, e.g. if $K_\cF$ is a parahoric subgroup attached to a facet $\cF \subset \cA$, then $K_\cF^{\flat}$ is still a group.

When $M^{\flat}=M_1$, the resulting objects will sometimes be described as "parahoric" and the supscript ${\flat}$ will be simply omitted, e.g. $I^{\flat}=I$ the Iwahori subgroup, $K_\cF^{\flat}=K_\cF$ the parahoric subgroup attached to a facet $\cF$, and so on ... 
When $M^{\flat}=M^1$, the resulting objects will be described as "geometric" and ${\flat}$ will be replaced simply by $1$, e.g. $I^{1} =\widetilde{I},K^{1}= \widetilde{K}$.

If $A$ is a ring and $H\subset \G(F)$ an open compact subgroup, define $\cH_H(A)$ to be the $A$-algebra of functions $f \colon \G(F) \to A$, that are $H$-invariant on the right and the left.

\subsubsection{Integrality of $\mathbb{U}^{\flat}$-operators}

In Corollary \ref{Zinjectintegralitytilde}, we obtain an embedding of rings $ \mathcal{R}^{-,\flat}\hra  \End_{\Z[B]}\Z[G/K^{\flat}]
$. 
Set $\mathbb{U}^{\flat}$ ({which stands for {\em "upper"}}) for its image, its elements are called $\mathbb{U}^{\flat}$-operators.
\begin{customtheorem}{\ref{integralitiyoperatorsu}}
For any ${u} \in \mathbb{U}^{\flat}$, there exists a monic polynomial $Q_u(X)=\sum {h}_k X^k \in \cH_{K^{\flat}}(\Z)[X]$ such that $u$ is a right root of $Q_u$, i.e.
$$\sum {h}_k\circ  u^k=0 \text{ in } \text{\emph{End}}_{\Z[B]}\Z[G/K^{\flat}].$$ 
\end{customtheorem}
A surprising aspect for proving this integrality result is that to establish formulas relating the two non-commuting commutative subrings, $\mathbb{U}^{\flat}$ and $\cH_{K^{\flat}}(\Z)$, of $\End_{\Z[B]}\Z[G/K]$, we embed them both in yet another non-commutative ring (the Iwahori--Hecke algebra), where they actually do commute. 
This passage is achieved using the compatibility of the (twisted) Bernstein and Satake isomorphism \cite[Theorem \ref{compatibility}]{UoperatorsI2021}.
 
\subsubsection{Geometric interpretation and Filtrations}

In \S \ref{retractionsanduoperators}, we translate the purely group theoretic ring $\mathbb{U}^{1}$ into a more combinatorial fashion. 
This will provide a new class of geometric operators on the set of special vertices in the extended building, which can be thought of as a conceptual generalization of the successor operators for trees with a marked end.

In \S \ref{filtrationandU}, we present yet another alternative geometric point of view for the geometric operators ring $\mathbb{U}^{1}$ using the notion of filtrations, this was suggested by C. Cornut.

\subsection{Further applications}

Upon completion of this paper, the author discovered another "incarnation" of the geometrico-combinatorial $\mathbb{U}$-operators which he was not aware of. 
This incarnation arises (under the name $U_p$ operator) in various places in the theory of automorphic forms: For example, this operator appears in the study of congruences between values of the partition function $p(n)$ \cite{AB2003,citeFKO2012},  
in the study of class polynomials \cite{AO2005} and in the theory of harmonic Maass forms \cite{Ono2008}. 

The properties of the $U_p$-operator play a crucial function both in Hida theory to define the ordinary projector \cite{Hida1986} %(overconvergent $p$-adic modular forms of slope zero)
and in Coleman and Mazur theory of overconvergent $p$-adic modular forms of finite slope \cite{colemanmazur1998} (the $\GL_2$ case), which has been later generalized to more general classes of eductive groups by Buzzard \cite{Buzzard2004}, Chenevier \cite{chenevier2004,Chenevier2011} and Urban \cite{Urban2011}.

{The author hopes that the perspective of this paper can provide more insight into the applications mentioned above.}
\subsection{Funding}
The author was supported by the Swiss National Science Foundation grant \#PP00P2-144658 and \#P2ELP2-191672.
\subsection{Acknowledgements}
This work is based partially on results proved in the author’s thesis, directed by D. Jetchev, to whom I am very grateful. I am also thankful to C. Cornut for his support and meticulous reading.

\newpage
\section{The ring of \texorpdfstring{$\mathbb{U}$}{U}-operators and its action}\label{ALring}

\subsection{Preliminary results}

Let $\Psi$ be a closed\footnote{A subset $\Psi \subset\Phi_+ \cap \Phi_{\red}$ such that for any $\alpha,\beta \in \Psi$ one has $[\alpha,\beta]:=\{n\alpha+m \beta\colon \text{ for all }n,m \in\Z_{>0} \}\cap \Phi \subset \Psi$.} 
subset of $\Phi_{\red}^+:=\Phi^+ \cap \Phi_{\red}$. 
There exists a unique closed, connected, unipotent $F$-subgroup $\U_\Psi\subset \G$, that is normalized by $\Mbf$ and such that the product morphism
\begin{equation}\label{isoPsi}\prod_{\alpha\in \Psi\cap \Phi_{\red}} \U_\alpha\to \U_\Psi,\tag{$\ast$}\end{equation}
is an isomorphism of $F$-varieties, here the product is taken in any order for $\Psi$ (see for example \cite[\S 14.4]{Bor91}).%In particular, $\text{Lie}_k(\U_\Psi)=\sum_{\alpha\in \Psi} \mathfrak{g}_\alpha$.
\begin{lemma}\label{UKUKtilde}
Let $\cF$ be any facet in $ \cA$. 
For any closed subset of roots $\Psi\subset \Phi_{\red}^+$, we have 
$$ U_{\pm \Psi}\cap K_\cF^{\flat}= U_{\pm\Psi} \cap U_\cF= \prod_{\alpha \in \pm \Psi}^{\prec}U_{\alpha+f_\cF(\alpha)},$$
for any fixed ordering $\prec$ on $\Psi$.
\end{lemma}
\begin{proof}
By \cite[Proposition \ref{propertiesII3.3} \ref{intersectionUpomega}]{UoperatorsI2021}, % and Example \ref{exfomega}
we have $U_{\alpha} \cap K_\cF= U_{\alpha+f_\cF(\alpha)}$ for any $\alpha \in \Phi_{\red}$. 
Combining this with the isomorphism (\ref{isoPsi}) and with \cite[Proposition \ref{decomppara}]{UoperatorsI2021}, we obtain that 
$$U_{\Psi}\cap {K}= \prod_{\alpha \in \Psi}^{\prec}U_{\alpha+f_\cF(\alpha)}$$
for any fixed ordering $\prec$ on $\Psi$.
\end{proof}

\begin{lemma}\label{sameU+}\label{BKI}
Let $\cF$ be a facet of $\a$ that contains in its closure $a_\circ$. 
For any closed subset of roots $\Psi\subset \Phi_{\red}^+$, we have
$$K_\cF^{\flat}\cap U_\Psi=K\cap U_\Psi.$$
In particular, for $\Psi= \Phi_{\red}^+$, we get $I^+=U_\cF^+=U_{a_\circ}^+$.
\end{lemma}
\begin{proof}
This is a direct consequence of Lemma \ref{UKUKtilde} and \cite[Example \ref{exfomega}]{UoperatorsI2021}.
\end{proof}

\subsection{Iwasawa decomposition for Levi subgroups}

\begin{proposition}[Iwasawa decomposition]\label{Iwasawadec}
If $\P=\mathbf{L}\ltimes \U_P^+$ is any standard parabolic subgroup of $\G$ with Levi factor $\mathbf{L}$ and unipotent radical $\U_P^+$, then $G=P K$, i.e. $K$ is a good open compact subgroup of $G$.    
Moreover, for any facet $\cF\subset \cA$ one has $P\cap K_\cF=(L\cap K_\cF) \cdot (U_P^+\cap K_\cF)$ and
$$L\cap K_\cF=(U_\cF^+\cap L) \cdot (U_\cF^- \cap L)\cdot (N_{1, \cF}\cap L)= (U_\cF^+\cap L) \cdot (U_\cF^- \cap L)\cdot (U_\cF^+\cap L)\cdot  M_1.$$
\end{proposition}
\begin{proof}
We refer the reader to \cite[Corollary \ref{Iwasawadeccor}]{UoperatorsI2021} for the Iwasawa decomposition and to \cite[\S 6.8]{HV15} for the factorization statement.\end{proof}

\subsection{Some parahoric open compact subgroups}
Let $\cF$ be a facet of $\a$ containing in its closure $a_\circ$, i.e. $I \subset K_\cF \subset K$. 
By \cite[4.6.33, 5.1.32]{BT84}, the parahoric subgroup $K_\cF$ is the inverse image of a standard parabolic subgroup of $\cG_{a_\circ}(\kappa({F}))$. We put ${W}_{\cF}^{\flat}:= N_{1, \cF}^{\flat}/M^{\flat}\simeq {W}_{\cF}$ where $N_{1, \cF}^{\flat}$ denotes the group $N \cap K_\cF^{\flat}= M^{\flat} N_{1,\cF}$. 

\begin{lemma}
Let $\cF \subset \cA$ be a facet containing in its closure a special point $a$, i.e. $ {K}_{\cF} \subset K_a$. 
The $\flat$-parahoric subgroup $K_{\cF}^\flat $ is equal to the subgroup generated by $M^\flat$ and $U_{\cF}$, $U_{\cF}$ is the subgroup of $G$ generated by $\cup_{\alpha \in \Phi_{\red}}U_{\alpha+f_{\cF}(\alpha)}$. We refer the reader to \cite[\S \ref{affinerootssection}]{UoperatorsI2021} for the definition of $f_\cF$. 
More precisely, we have 
$${K}_{\cF}^\flat= K_{\cF} M^\flat= U_{\cF} M^\flat =U_{\cF}^+ U_{\cF}^- U_{\cF}^+ M^\flat=U_{\cF}^+ U_{\cF}^- N_{1, \cF}^{\flat}  \text{ and }K_\cF^\flat/{K}_{\cF}\simeq M^\flat/M_1,$$
such that the factors commute in the right factorization and the product maps 
$$\begin{tikzcd}\prod_{\alpha \in \Phi_{\red}\cap \Phi^\pm} U_{\alpha+f_{{\cF}}(\alpha)}\arrow{r}{\sim}&U_{\cF}^\pm=U_{\cF}\cap \U(F)^\pm=K_{\cF}^\flat\cap \U(F)^\pm.\end{tikzcd}$$
is an homeomorphism, whatever ordering of the factors we take. 
\end{lemma}
\begin{proof}
This is Proposition \ref{sectiondecomp} and Lemma \ref{G1fixatorP} of \cite{UoperatorsI2021}.
\end{proof}

By \cite[XXVI 7.15]{sga3.3}, for any parabolic subgroup $\cP\subset \cG$, the base change $\cP_{\Spec\, \kappa(F)}$ is also a parabolic subgroup of $\cG_{\Spec\, \kappa(F)}$. 
Conversely, any parabolic subgroup $\overline{\cP} \subset \cG_{\Spec\, \kappa(F)}$ there exists a parabolic subgroup $\cP \subset  \cG$ such that $\cP_{\Spec\, \kappa(F)}=\overline{\cP}$. 
In particular, a parabolic subgroup $\cP\subset  \cG$ is minimal if and only if $\cP_{\Spec\, \kappa(F)}\subset \cG_{\Spec\, \kappa(F)}$ is minimal.

For any $J \subset \Delta$, set $\Sbf_J:=(\cap_{\alpha \in J} \ker \alpha)^\circ$. 
We get a parabolic subgroup $\mathbf P_J= \mathbf{L}_J \ltimes \U_J$, 
whose Levi factor is $\mathbf L_J:= Z_\G(\Sbf_J)$ and unipotent radical $\mathbf U_J:=\U_{\Phi^+ \setminus  \Phi_J} $, 
where $\Phi_J=\Phi\cap \Z J$.

Therefore, there exists a standard parabolic subgroup $\P_\cF=\P_{J_\cF}= \Lbf_\cF \rtimes \U_{\Phi_{\red}^+\setminus \Phi_{J_\cF}}$ corresponding to a subset $J_\cF\subset \Delta$ attached to $\cF$ such that 
\begin{equation}\label{parahoricisparabolic}K_\cF = \{ g \in \cG_{a_\circ} (\cO_{F})\colon \red(g  ) \in \P_\cF(\kappa(F))\}, \tag{$\star$}\end{equation}
where $"\red"$ denotes the reduction map $ \cG_{a_\circ} (\cO_{F}) \to \cG_{a_\circ} (\kappa(F))$.

\begin{theorem}\label{decompositionwithLevi} Let $\cF$ be a facet of $\a$ containing $a_\circ$ in its closure. Therefore
\begin{enumerate}[1 -]
\item  $f_\cF(\alpha)=0$ for all $\alpha \in \Phi_{\red}^+  \cup  \Phi_{J_\cF} $, $f_\cF(\alpha)=n_\alpha^{-1}$ for all $\alpha \in \Phi_{\red}^- \setminus \Phi_{J_\alpha}$,
\item $W_\cF\subset W$ is the subgroup generated by $s_\alpha$ for $\alpha \in J_\cF$, thus it equals the Weyl group of $\Lbf_\cF$, in particular $N_{1, \cF}$ centralizes $ \U_{\alpha}$ for all $\alpha \in \Phi_{\red} \setminus \Phi_J$,
\item $L_\cF \cap K_\cF= L_\cF \cap K$, i.e. this is a special maximal parahoric subgroup of $L_\cF$, in particular $P_\cF \cap K_\cF= P_\cF \cap K$.
\item $K_\cF=U_{J_\cF}^+ \cdot (L_\cF \cap K_\cF) \cdot  U_{J_\cF}^-=U_{J_\cF}^-\cdot (L_\cF \cap K_\cF) \cdot  U_{J_\cF}^+,$ 
where, $U_{J_\cF}^\pm:=U_{{  \Phi_{\red}^\pm\setminus \Phi_{J_\cF}}} \cap K_\cF$. 
\end{enumerate}
\end{theorem}
When $\cF=\a$ then $J_\a=\emptyset$, $\Lbf_\a=\Mbf$ and the above decomposition is the usual one $I= I^+ M_1 I^-$.
\begin{proof}
\begin{enumerate}[1 -]
\item This follows from \cite[Example \ref{exfomega}]{UoperatorsI2021}, since by (\ref{parahoricisparabolic}), we have 
$$U_{\alpha+0}=\U_\alpha(F)\cap K_\cF= \U_\alpha(F)\cap K$$
for all $\alpha \in (\Phi^+\cap \Phi_{\red})  \cup  \Phi_{J_\cF} $.
\item By \cite[Lemma \ref{NactionUalphar}]{UoperatorsI2021}, an element $n\in N \cap K$ lies in $N_{1, \cF}$ if and only if
$f_\cF(\alpha) = f_\cF(w_n(\alpha))$ for all $\alpha \in \Phi_{\red}$, where $w_n$ denotes the image of $n$ in $W$. But since for all $\alpha \in \Phi_{\red}^+\cup \Phi_J$ one has $f_\cF(\alpha) =0$, we deduce that such $n$ lies in $N_{1, \cF}$ if and only if $f_\cF(\alpha) = f_\cF(w_n(\alpha))=0$ for all $\alpha \in \Phi_{\red}^+\cup \Phi_J$ or equivalently $w_n(\Phi_{\red}^+\cup \Phi_J)=\Phi_{\red}^+\cup \Phi_J$.

Therefore, $W_\cF\subset W$ is the subgroup generated by $s_\alpha$ for $\alpha \in \Phi_{\red}^+$ such that $f_\cF(-\alpha)=0$, i.e. $\alpha \in J_\cF$. 
This implies that $W_\cF$ identifies with the Weyl group $W_L$ of $\Lbf_\cF$ and consequently $$N_{1, \cF} \subset  L_\cF \text{ and }   N_{1, \cF}/M_1 \simeq  N_{ L_\cF}(S)/M=W_L .$$
In particular, $W_\cF$ fixes all $\Phi_{\red} \setminus \Phi_J$ hence $N_{1,\cF}$ centralizes each $\U_\alpha$ for $\alpha \in \Phi_{\red} \setminus \Phi_J$.
\item As observed in (1), we have  by definition $\U_\alpha(F)\cap K_\cF= \U_\alpha(F)\cap K$  for all $\Phi_{\red}^+   \cup  \Phi_{J_\cF, \text{red}}$. Therefore, using Proposition \ref{Iwasawadec} for $\cF$ and $a_\circ$ and Lemma \ref{BKI} for $\Psi= \Phi_{\red}^+\setminus \Phi_{J_\cF}$, it remains to shows that $N_{1, \cF} \cap \Lbf_\cF= N \cap K \cap \Lbf_\cF$. But since one has an inclusion $N_{1, \cF} \cap \Lbf_\cF /M_1\subset  N \cap K \cap \Lbf_\cF M_1 \subset W_L$ the equality follows from (2).
\item By \cite[Proposition \ref{decomppara}]{UoperatorsI2021} we know that $K_\cF=U_\cF^+ U_\cF^- N_{1, \cF}$. We have $\U^\pm =\U_{\Lbf_\cF} \U_{\Phi^\pm \setminus  \Phi_{J_\cF}}$ and accordingly
$$K_\cF=U_\cF^+ U_\cF^- N_{1, \cF}= (U_{{  \Phi_{\red}^+\setminus \Phi_{J_\cF}}} \cap K_\cF)  (L_\cF \cap U_\cF^+) (U_\cF^-\cap L_\cF) (U_{{  \Phi_{\red}^-\setminus \Phi_{J_\cF}}} \cap K_\cF) N_{1, \cF}.$$
By (2), $N_{1, \cF}$ centralizes $(U_{{  \Phi^-\setminus \Phi_{J_\cF}}} \cap K_\cF) $ and $(U_{ {  \Phi^+\setminus \Phi_{J_\cF}}} \cap K_\cF) $, hence using Proposition \ref{Iwasawadec} one obtains
$$K_\cF = (U_{{  \Phi_{\red}^+\setminus \Phi_{J_\cF}}} \cap K_\cF)  (L_\cF \cap K_\cF) (U_{{  \Phi_{\red}^-\setminus \Phi_{J_\cF}}} \cap K_\cF) = (U_{{  \Phi_{\red}^-\setminus \Phi_{J_\cF}}} \cap K_\cF)  (L_\cF \cap K_\cF) (U_{ {  \Phi_{\red}^+\setminus \Phi_{J_\cF}}} \cap K_\cF). \qedhere $$ 
\end{enumerate}
\end{proof}
\begin{lemma}\label{bijection42}
For any $J\subset \Delta$ and any $m \in M_J^-$, the inclusion $U_J^+ \cap {K} \hra I^+=U^+ \cap {K}$, induces the bijection
$$\begin{tikzcd}U_{J}^+ \cap {K}_\cF^{\flat}/ U_{J}^+\cap m{K}_\cF^{\flat} m^{-1} \arrow{r}{\simeq}&I^+/mI^+m^{-1}.\end{tikzcd}$$
\end{lemma}
For any $J\subset \Delta$ and any $m \in M_J^-$, one has $U_{J}^+\cap K= \prod_{\alpha \in \Phi^+ \setminus \Phi_J}^{\prec}U_{\alpha+0}$ for any fixed ordering $ \prec$ of $\Phi^+ \setminus \Phi_J$.
\begin{proof}
We have $\U^+= \U_{J}^+ (\U^+\cap \mathbf{L}_{J})$, so by Lemma \ref{UKUKtilde} we get a decomposition
$$U^+\cap K_\cF^{\flat}=(U^+ \cap L_{J} \cap K_\cF)(U_{J}^+\cap K_\cF)=(U_{J}^+\cap K_\cF)(U^+ \cap L_{J} \cap K_\cF).$$
By \cite[Lemma \ref{NactionUalphar}]{UoperatorsI2021}, we have  $m U_{\alpha+0}m^{-1}= U_{\alpha- \langle \alpha , \nu(m)\rangle}$, for all $\alpha \in \Phi_{\red}$. 
But since $\langle \alpha , \nu(m)\rangle=0$ for all $\alpha \in \Phi(\mathbf{L}_{J },\Sbf)$, %\footnote{By definition, the semi-standard Levi $\mathbf{L}_{m}$ centralizes $m$.}
we see that $U^+\cap mK_\cF^{\flat}m^{-1}=(U_{J}^+\cap mK_\cF m^{-1})(U^+ \cap L_{J} \cap K_\cF)$. 
Therefore, for any $m\in M_\cF^-$, we have
$$U^+\cap K_\cF^{\flat}/U^+\cap mK_\cF^{\flat}m^{-1}\simeq U_{J}^+\cap K_\cF/ U_{J}^+ \cap m K_\cF m^{-1}$$
Which shows the lemma since by Lemma \ref{BKI}
$U^+\cap K_\cF^{\flat}=U_\cF^+= I^+$ and $U^+\cap mK_\cF^{\flat}m^{-1}=mI^+m^{-1}$.
\end{proof}

\subsection{Hecke algebras}
Let $H$ be any open compact subgroup of $G$. 
For any commutative ring $A$, we define $\cC_c(G/H,A)$ to be the $A$-module of compactly supported functions $f \colon G \to A$ that are $H$-invariant on the right. 
It has the following canonical basis $\{{\bf 1}_{gH} \colon g \in G/H\}$. 

We also define $\cH_c(G\sslash H,A)\subset \cC_c(G/H,A)$ to be the convolution $A$-algebra of functions $f \colon G \to A$, that are also $H$-invariant on the left. 
The convolution is defined with respect to the Haar measure giving $H$ volume $1$ and will be denoted $*_H$.
We call $\cH_c(G\sslash H,A)$ the Hecke algebra relative to $H$ with values in $A$ and denote $\cH_H(A)$. 
We have $\cH_H(A)=\cH_H(\Z)\otimes_{\Z} A$. 

\begin{example}
When $\cF=\{a_\circ\}$, $\cH_{K^{\flat}}(\Z)$ will be called the ${\flat}$-special--Hecke algebra. 
One can exhibit a natural $\Z$-basis for $\cH_{K^{\flat}}(\Z)$ as follows $\{{\bf 1}_{K^{\flat} mK^{\flat}} \text{ for } m \in \Lambda_M^{-,{\flat}}\}%=\{i_{\varpi^\mu w}:={\bf 1}_{I{\varpi^\mu w}I} \text{ for } {\varpi^\mu  w}\in \Lambda_M\rtimes W\}.
$. 
When $\cF=\a$, $\cH_{I^{\flat}}(\Z)$ is called the ${\flat}$-Iwahori--Hecke algebra and the following set forms a $\Z$-basis $\{{\bf 1}_{I^{\flat}wI^{\flat}} \text{ for } w\in N/M^\flat \}.%=\{i_{\varpi^\mu w}:={\bf 1}_{I{\varpi^\mu w}I} \text{ for } {\varpi^\mu  w}\in \Lambda_M\rtimes W\}.
$
\end{example}
\subsection{Parabolic partition of $\Lambda_M^-$}
We partition $\Lambda_M^-$ into facets corresponding to parabolic subgroups containing $\B$. These parabolics are classified by arbitrary subsets $J\subset \Delta$. 
Define  
$M_\cF^-:=\{m \in M \colon \langle\nu_M(m), \alpha \rangle \ge 0, \forall \alpha \in \Phi^- \cup \Phi_{J_\cF}\}$ and $\Lambda_{M,J_\cF}^- \subset \Lambda_M^-$ its image in $\Lambda_M$. 
When $\cF=\a$, i.e. $J_\cF=\emptyset$, we get $\mathbf P_\emptyset = \B$ and $M_\emptyset^-=M^-$.
We obtain a partition
$${\Lambda}_M^-= \bigcup_{J_\cF \subset \Delta} \Lambda_{M,J_\cF}^-.$$
For any $n\in N$, set $h_{n,\cF}^\flat:={\bf 1}_{K_\cF n K_\cF}$, when $\cF=\a$, $i_{n}^\flat:= h_{n,\cF}^\flat$ and put $\mathcal{R}_{\cF}^{-,\flat}$ for the ring $\Z[\Lambda_{M,J_\cF}^{-,\flat}]$.
\begin{lemma}\label{facetaddit}
The map of $\Z$-modules $$\jmath_\cF^\flat \colon \mathcal{R}_{\cF}^{-,\flat} \hra \cH_{K_\cF^\flat}(\Z), \quad m \mapsto h_{m,\cF}^\flat$$
is an embedding of rings. 
\end{lemma}
\begin{proof}
We need to show that for any $m,m' \in \Lambda_{M,J_\cF}^{-,\flat}$, we have
$$h_{m,\cF}^\flat*_{K_\cF^\flat}h_{m',\cF}^\flat=h_{m+m',\cF}^\flat.$$
The function $h_{m,\cF}^\flat*_{K_\cF^\flat}h_{m',\cF}^\flat$ is supported on $K_\cF^\flat mK_\cF^\flat m' K_\cF^\flat$, which is equal by 4 - Theorem \ref{decompositionwithLevi} to 
$K_\cF^\flat m U_{J_\cF}^+  (L_\cF \cap K_\cF)   U_{J_\cF}^- m'K_\cF^\flat$. 
By \cite[Lemma \ref{normalizingI}]{UoperatorsI2021}, since $m\in M_\cF^-$, we have $mU_{J_\cF}^+m^{-1} \subset U_{J_\cF}^+ $ and $m^{-1}U_{J_\cF}^-m\subset U_{J_\cF}^-$. 
Therefore, 
$$h_{m,\cF}^\flat*_{K_\cF^\flat}h_{m',\cF}^\flat = s h_{m+m',\cF}^\flat, \text{ for }s=\big|m^{-1}K_\cF^\flat mK_\cF^\flat \cap m'K_\cF^\flat m'^{-1}K_\cF^\flat \big|_{K_\cF^\flat}.$$
Combining 4 - Theorem \ref{decompositionwithLevi} and  \cite[Lemma \ref{normalizingI}]{UoperatorsI2021} again, one shows that
$$m^{-1}K_\cF^\flat mK_\cF^\flat \cap m'K_\cF^\flat m'^{-1}K_\cF^\flat= (m^{-1}U_{J_\cF}^+ m) (L_\cF \cap K_\cF^\flat)U_{J_\cF}^- \cap  (m'U_{J_\cF}^- m'^{-1})(L_\cF \cap K_\cF^\flat) U_{J_\cF}^+=K_\cF^\flat.$$
This concludes the proof of the lemma.
\end{proof}

\subsection{Section}
The inclusion $K_\cF^{\flat}\subset K^{\flat}$ induces a surjective map
$$\begin{tikzcd}
\cC_c(G/K_\cF^{\flat},\Z)\arrow[r, twoheadrightarrow]&\cC_c(G/K^{\flat},\Z), &f \mapsto f*_{K_\cF^{\flat}}{\bf{1}}_{K^{\flat}}.
\end{tikzcd}$$
The same inclusion induces also natural embeddings
$$\begin{tikzcd}
\cC_c(G/K^{\flat},\Z)\arrow[r, hook]&\cC_c(G/K_\cF^{\flat},\Z), &\cH_{K^{\flat}}(\Z)\arrow[r, hook]&\cH_{K_\cF^{\flat}}(\Z).
\end{tikzcd}$$
The composition of the above maps
$$\begin{tikzcd}
\cC_c(G/K^{\flat},\Z)\arrow[r, hook]&\cC_c(G/K_\cF^{\flat},\Z)\arrow[r, twoheadrightarrow]&\cC_c(G/K^{\flat},\Z),
\end{tikzcd}$$
is multiplication by $[K^{\flat}:K_\cF^{\flat}]=[K:K_\cF]$.
We fix a section of the natural projection $\pi_\cF\colon G/K_\cF^{\flat}\surjto G/K^{\flat}$, as follows:
 $$\begin{tikzcd}s_\cF\colon G/K^{\flat} \arrow{r}{\simeq}& P_\cF/(P_\cF\cap K^{\flat})\arrow[r, hook] &G/K_\cF^{\flat},\end{tikzcd}$$
the bijectivity of the first map follows from the Iwasawa decomposition $G=P_\cF K^{\flat}$ \cite[{Proposition \ref{Iwasawadec}}]{UoperatorsI2021}, 
while the injectivity of the second map follows from 
3 - Theorem \ref{decompositionwithLevi}.
The above map induces a map $\begin{tikzcd}\cC_c(G/K^{\flat},\Z)\arrow{r}&\cC_c(G/K_\cF^{\flat},\Z)
\end{tikzcd}$, also denoted by $s_\cF$ 
defined on the basis functions by ${\bf 1}_{bK^{\flat}}\mapsto {\bf 1}_{bK_\cF^{\flat}}$, for all $b\in P_\cF$. 
The map $s_\cF$, is actually a section of $-*_{K_\cF^{\flat}} {\bf 1}_{K^{\flat}}\colon  \cC_c(G/K_\cF^{\flat},\Z) \twoheadrightarrow  \cC_c(G/K^{\flat},\Z)$.
\begin{remark}\label{functo1}
We have a commutative diagram
$$\begin{tikzcd} G/K^{\flat} \arrow[equal]{d}\arrow{r}{\cong}& P_\cF/(P_\cF\cap K^{\flat})\arrow[r, hook] &G/K_\cF^{\flat}\\
 G/K^{\flat} \arrow{r}{\simeq}& B/(B\cap K^{\flat})\arrow{u}{\cong}\arrow[r, hook] &G/I^{\flat}\arrow[two heads]{u}
 \end{tikzcd}$$
So we have 
 $s_\cF =(- *_{I^{\flat}} {\bf 1}_{K_\cF^{\flat}})\circ s_\a $. \end{remark}
 \begin{remark}
On the level of the extended building, the injection $G/K^1 \hookrightarrow G/I^1$ induces a $B$-equivariant embedding of the $G$-orbit of $(a_{\circ}, 0_{V_G})$ into the $G$-orbit of $\a \times 0_{V_G}$.
\end{remark}

\subsection{A \texorpdfstring{$\mathcal{R}_{\cF}^{-,\flat}$}{R}-action on $\Z[G/K]$}\label{pairingsection}
We define an "excursion pairing"
$$\begin{tikzcd}
\cC_c(G/K^{\flat},\Z)\times\cH_{K_\cF^{\flat}}(\Z) \arrow{r}&\cC_c(G/K^{\flat},\Z);& (x,f)\arrow[mapsto]{r}&x\bullet_\cF f:= \left(s_\cF(x)*_{K_\cF^{\flat}} f\right)*_{K_\cF^{\flat}} {\bf 1}_{K^{\flat}}.
\end{tikzcd}$$
This is clearly bilinear in both variables. 

\begin{lemma}[{$\mathcal{R}_{\cF}^{-,\flat}$ action on $\Z[G/K^{\flat}]$}]\label{Uaction}
The "excursion pairing" when restricted to $\jmath_\cF^{\flat}(\mathcal{R}_{\cF}^{-,\flat})$ defines a right action
$$\begin{tikzcd}[column sep= tiny]
\cC_c(G/K^{\flat},\Z)&\times_{\jmath_\cF^\flat}&\mathcal{R}_{\cF}^{-,\flat}\arrow{rrrrr}&&&&&\cC_c(G/K^{\flat},\Z).
\end{tikzcd}$$
\end{lemma}
\begin{proof}
Let $b\in P_\cF$ be any representative of $gK^{\flat}$. Combining 4 - Theorem \ref{decompositionwithLevi} and \cite[Lemma \ref{normalizingI}]{UoperatorsI2021}, we have
$$K_\cF^\flat m K_\cF^\flat = U_{J_\cF}^+ m K_\cF^\flat \text{ and } mU_{J_\cF}^+ m^{-1} \subset U_{J_\cF}^+\text{ for all }m\in M_\cF^-,$$
Then, for $x={\bf 1}_{pK^{\flat}}$ and $f={\bf 1}_{K_\cF^{\flat} mK_\cF^{\flat}}$, we get %(using \cite[Lemma \ref{actionhecke}]{UoperatorsI2021})
$$x\bullet_\cF f = {\bf 1}_{pU_{J_\cF}^+ mK_\cF^{\flat}}*_{K_\cF^{\flat}}{\bf 1}_{K^{\flat}}%=\sum_{i \in I/ I\cap gIg^{-1}} {\bf 1}_{bigI}*_I{\bf 1}_{K}
=\sum_{u \in U_{J_\cF}^+/  mU_{J_\cF}^+ m^{-1}} {\bf 1}_{pumK^{\flat}}.$$
Accordingly, for any $m, m' \in \Lambda_{M,J_\cF}^{-,\flat}$, we then have
\begin{align*}
(\mathbf{1}_{gK^{\flat}}\bullet_\cF h_{m,\cF}^{\flat} )\bullet_\cF h_{m',\cF}^{\flat}&=(\sum_{u \in U_{J_\cF}^+/ U_{J_\cF}^+\cap mU_{J_\cF}^+ m^{-1}} {\bf 1}_{pumK^{\flat}})\bullet_\cF h_{m',\cF}^{\flat}\\
&=(\sum_{u \in U_{J_\cF}^+/ U_{J_\cF}^+\cap mU_{J_\cF}^+ m^{-1}} {\bf 1}_{pumK_\cF^{\flat}}  *_{K_\cF^{\flat}} h_{m',\cF}^{\flat})*_{K_\cF^{\flat}}\mathbf{1}_{K^{\flat}}\\
&=( {\bf 1}_{pK_\cF^{\flat}}*_{K_\cF^{\flat}} h_{m,\cF}^{\flat}*_{K_\cF^{\flat}} h_{m',\cF}^{\flat})*_{K_\cF^{\flat}}\mathbf{1}_{K^{\flat}}\\
&\overset{\text{Lem. \ref{facetaddit}}}{=} {\bf 1}_{pK_\cF^{\flat}}*_{K_\cF^{\flat}} h_{m+m',\cF}^{\flat}*_{K_\cF^{\flat}}\mathbf{1}_{K^{\flat}}\\
&= {\bf 1}_{pK^{\flat}} \bullet_\cF h_{m+m',\cF}^{\flat}.\quad\qedhere
\end{align*} 
\end{proof}

\begin{remark}\label{functo216}
From now on, for any $m\in M_\cF^-$ we denote $x\bullet_\cF {\bf 1}_{K_\cF m K_\cF}$ simply by $x\bullet_\cF m$. 
By Remark \ref{functo1}, for any $m\in M_\cF^-$ and any $x\in \cC_c(G/K^{\flat},\Z)$, we have
$$x\bullet_\cF m = \left(s_\cF(x)*_{K_\cF} {\bf 1}_{K_\cF m K_\cF} \right)*_{K_\cF} {\bf 1}_{K}=\left(s_\a(x)  *_{I}  {\bf 1}_{K_\cF m K_\cF} \right)*_{K_\cF} {\bf 1}_{K}.$$ 
But, by \cite[Lemma \ref{normalizingI}]{UoperatorsI2021}, we have $K_\cF m K_\cF=I m K_\cF$, for $m\in M_\cF^-$. Therefore, 
$$x\bullet_\cF m =\left(s_\a(x)  *_{I}  {\bf 1}_{I m I} \right) *_{I}  {\bf 1}_{K_\cF}*_{K_\cF} {\bf 1}_{K}= x\bullet_\a {\bf 1}_{I m I}=x\bullet_\a m .$$
We can also argue explicitly; 
we saw that for $x={\bf 1}_{pK^{\flat}}$ $(b \in B)$ and $f={\bf 1}_{K_\cF^{\flat} mK_\cF^{\flat}}$ $(m\in M_\cF^-)$, so by Lemma \ref{bijection42} one shows%(using \cite[Lemma \ref{actionhecke}]{UoperatorsI2021})
$$x\bullet_\cF m =\sum_{u \in U_{J_\cF}^+/ mU_{J_\cF}^+ m^{-1}} {\bf 1}_{pumK^{\flat}}=\sum_{u \in I^+/ mI^+ m^{-1}} {\bf 1}_{pumK^{\flat}}=x\bullet_\a m.$$
We may then (and will) omit the subscript $\cF$, and just write $x\bullet m$ viewing $m$ in $M^-$. 
In conclusion, we have a commutative diagram
$$\begin{tikzcd}[column sep= tiny]
\cC_c(G/K^{\flat},\Z)\arrow[equal]{d}&\times_{\jmath_\cF^\flat}&\mathcal{R}_{\cF}^{-,\flat} \arrow[hook]{d}\arrow{rrrrr}&&&&&\cC_c(G/K^{\flat},\Z)\arrow[equal]{d}\\
\cC_c(G/K^{\flat},\Z)&\times_{\jmath_\a^\flat}&\mathcal{R}_{\a}^{-,\flat} \arrow{rrrrr}&&&&&\cC_c(G/K^{\flat},\Z).
\end{tikzcd}$$
\end{remark}

\begin{remark}\label{compatibilitybis}
Recall that \cite[Proposition \ref{ringintertwiners}]{UoperatorsI2021} gives another interpretation of the parahoric--Hecke algebra $\cH_{K_\cF^{\flat}}(\Z)$, namely:
$$\begin{tikzcd}[row sep= tiny]\cH_{K_\cF^{\flat}}(\Z) \arrow{r}{\simeq}& \Z[K_\cF^{\flat} \backslash G /K_\cF^{\flat}] \arrow{r}{\simeq}&( \End_{\Z[G]}(\Z[G/K_\cF^{\flat}]))^{\text{opp}},\end{tikzcd}$$
here, the superscript {\em opp} indicates the opposite ring. 
when $\cF=a_\circ$, the resulting Hecke algebra is commutative, so we may drop the superscript "opp". 
We also have an obvious canonical identification $\Z[G/K_{\cF}]\simeq  \cC_c(G/K_{\cF},\Z)$, which intertwines the left action of $\End_{\Z[G]}\Z[G/K_{\cF}]$ on $\Z[G/K_{\cF}]$ with the right action of $\cH_{K_{\cF}}(\Z)$ on $\cC_c(G/K_{\cF},\Z)$. 
For any $m\in M_\cF^-$, we also write $h_{\cF,m}^\flat$ (by abuse of notation) for the element in $\End_{\Z[G]}\Z[G/K_{\cF}]$ corresponding to $\mathbf{1}_{K_\cF^\flat m K_\cF^\flat} \in \cH_{K_\cF^{\flat}}(\Z)$.

Using these identifications, the "excursion pairing" becomes
$$\begin{tikzcd}
\Z[G/K^{\flat}]\times (\End_{\Z[G]}(\Z[G/K_\cF^{\flat}])])^{\text{opp} }\arrow{r}& \Z[G/K^{\flat}];& (x,f)\arrow[mapsto]{r}&x \bullet_\cF f =\pi_\cF \circ f\circ s_\cF (x).
\end{tikzcd}$$

\end{remark}

\begin{lemma}\label{faithfulness}
The action of $\mathcal{R}_{\cF}^{-,\flat}$ on the $\Z$-module $\Z[G/K^\flat]\simeq \cC_c(G/K^{\flat},\Z)$ is faithful.
\end{lemma}
\begin{proof}For any  pairwise distinct $m_1,\cdots m_r \in \Lambda_{M,J_\cF}^{-,\flat}$ and $s_1, \cdots, s_r\in \Z$, we have
\begin{align*}
{\bf 1}_{K^\flat} \bullet (\sum_i s_i\,{m_i})
&= \sum_{i} s_i {\bf 1}_{K_\cF^{\flat} m_iK^{\flat}}=\sum_{i} s_i {\bf 1}_{I^{\flat} m_iK^{\flat}}.
\end{align*}
But, by \cite[Proposition 3.4.1]{UoperatorsI2021}, we know that $I^{\flat}\backslash G/K^{\flat} \cong \widetilde{W}^{\flat}/W_{a_\circ}^{\flat}\cong \Lambda_M^{\flat}$, and so all ${\bf 1}_{I^{\flat} m_iK^{\flat}}, m\in \Lambda_M^{\flat}$ are linearly independent. This shows the lemma.
\end{proof}
\begin{corollary}\label{injectU}\label{Zinjectintegralitytilde}\label{Pinvariance} The action of $\mathcal{R}_{\cF}^{-,\flat}$ on $\Z[G/K^{\flat}]$ induces an embedding of rings $$\begin{tikzcd}\varphi_\cF \colon \mathcal{R}_{\cF}^{-,\flat} \arrow[hook]{r}{}& \End_{\Z[P_\cF]}\Z[G/K^{\flat}], \quad m \mapsto \pi_\cF \circ h_{\cF,m}^\flat \circ s_\cF.\end{tikzcd}$$
%We don't really need the superscript $\text{opp}$ since the ring $\mathbb{U}$ is commutative. 
\end{corollary}
\begin{proof}
Lemma \ref{faithfulness} shows that we have an embedding of rings $ \mathcal{R}_{\cF}^{-,\flat} \hra \End_{\Z}\Z[G/K^{\flat}]$. Remark \ref{compatibilitybis} explicate why $\varphi_\cF(m)= \pi_\cF \circ h_{\cF,m}^\flat \circ s_\cF$, for any $m\in \Lambda_{M,J_\cF}^{-,\flat}$. 
The $P_\cF$-equivariance is a straightforward consequence of the formula
$${\bf 1}_{pK}\bullet_\cF m =\sum_{u \in U_{J_\cF}^+/ U_{J_\cF}^+\cap mU_{J_\cF}^+ m^{-1}} {\bf 1}_{pumK^{\flat}},$$
for $p\in P$ and $m\in \Lambda_{M,J_\cF}^{-,\flat}$.
\end{proof}

\subsection{Definition of $\mathbb{U}^{\flat}$-operators}

\begin{definition}\label{definuflat}
Define the ring of $\mathbb{U}_\cF^{\flat}$-operators to be the commutative subring
$$\varphi_\cF ( \mathcal{R}_{\cF}^{-,\flat})\subset  \End_{\Z[P_\cF]}\Z[G/K^{\flat}].% \simeq \End_{\Z[P_\cF]}\cC_c(G/K^{\flat},\Z).
$$
When $\cF=\a$, we simply write $\mathbb{U}^{\flat}$ and put ${u}_m^{\flat}:= \varphi_\cF(m) \in \mathbb{U}^{\flat}$, for any $m\in \Lambda_M^{-,\flat}$.
\end{definition}
\begin{remark}\label{commutflatuop}
(i) Although, $\End_{\Z[G]}\cC_c(G/K,\Z)$ and $\mathbb{U}_\cF^{\flat}$ are both commutative subrings of $\End_{\Z[P_\cF]}\cC_c(G/K^{\flat},\Z)$, they do not commute with each other. 
(ii) Thanks to Remark \ref{functo216}, we have a commutative diagram
$$\begin{tikzcd}
\mathbb{U}_{\cF}^{\flat} \arrow[hook]{d}\arrow[hook]{r}&\End_{\Z[P_\cF]}\Z[G/K^{\flat}]\arrow[hook]{d}\\
\mathbb{U}^{\flat}\arrow[hook]{r}&\End_{\Z[B]}\Z[G/K^{\flat}]
\end{tikzcd}$$
\end{remark}

\subsection{Integral Satake and Bernstein homomorphisms}
\subsubsection{The Weyl dot-action}
We define a twisted action of $W$ on $\Z[q^{-1}] \otimes_\Z \mathcal{R}^{\flat}$, is given on basis elements by 
$$\dot{w} (mM^{\flat}) = w(m)M^{\flat}\in \Lambda_M^{\flat}$$ 
wehre $c(m,w):= \left(\frac{\delta_B(w(m))}{\delta_B(m)}\right)^{1/2}$ and extended $\Z[q^{-1}]$-linearly to $\Z[q^{-1}] \otimes_\Z \mathcal{R}^{\flat}$.

We write $\dot{W}$ to specify that we are dealing with the dot-action of $W$ and not the standard one. 
The natural projection map ${\square}^{\flat}\colon \mathcal{R} \to \mathcal{R}^{\flat}$ is $\dot{W}$-equivariant, i.e. 
$(\dot{w}(m))^{\flat}= \dot{w}(m^{\flat})$ for any $w\in W$ and any $m\in \Lambda_M$. 
\begin{lemma}
If $m \in M^-$ then $c(m,w)\in q^{\Z_{\ge 0}}$, for any $w\in W$. Equivalently, $\dot{w}(\mathcal{R}^{-,\flat}) \subset \mathcal{R}^\flat$ for any $\dot{w} \in \dot{W}$.
\end{lemma}
\begin{proof}
This is \cite[Lemma \ref{crintegral} (ii)]{UoperatorsI2021}.
\end{proof}
\subsubsection{Universal unramified principal series}
Let $\cF \subset \cA$ be a facet containing in its closure a special point $a$. 
Define the universal unramified principal series right $\Z$-module $\cM_{K_{\cF}^{\flat}}(\Z) = \Z[M^{\flat}U^+\backslash G/K_{\cF}]$. 
It comes with a natural right $(\End_{\Z[G]}(\Z[G/K_\cF^{\flat}]))^{\text{opp}}$-module structure. This can be seen as the natural right action of $\cH_{K^\flat}(\Z)$ by convolution with respect to the normalized measure $\mu_{K_\cF^{\flat}}$ giving $K_\cF^{\flat}$ volume $1$. 
We write the action of an element $h \in \End_{\Z[G]}(\Z[G/K_\cF^{\flat}])$ on $v\in \cM_{K_{\cF}^{\flat}}(\Z) $ by $v^h$.

The family $$\{v_{mw,\cF}^{\flat}\colon (m,w) \in \Lambda_M^\flat \times ({N}_{1,a_\circ}^{\flat}/N_{1,\cF}^{\flat})\},$$ forms a $\Z$-basis for the $\Z$-module $\cM_{K_\cF^{\flat}}(\Z)$ \cite[Lemma \ref{basisMI}]{UoperatorsI2021}. 

We endow $\cM_{K_{\cF}^{\flat}}(\Z)$ with a left $\mathcal{R}^{\flat}$-action given on $\Z$-basis elements by: 
$$m'\cdot v_{,\cF}^{\flat} = v_{m'mw,\cF}^{\flat},\quad \text{ for any $m',m \in \Lambda_M^{\flat}$ and any $w \in {N}_{1,a_\circ}^{\flat}/N_{1,\cF}^{\flat}$}.$$
It is clear that the actions of $(\End_{\Z[G]}(\Z[G/K_\cF^{\flat}]))^{\text{opp}}$ and $\mathcal{R}^{\flat}$ on $\cM_{K_{\cF}^{\flat}}(\Z)$ commute: $\cM_{K_{\cF}^{\flat}}(\Z)$ is an $(\mathcal{R}^{\flat},(\End_{\Z[G]}(\Z[G/K_\cF^{\flat}]))^{\text{opp}})$-bimodule.
The following proposition gives a precise description of the $\mathcal{R}^{\flat}$-module $\cM_{K_\cF^{\flat}}(\Z)$. 
\begin{proposition}\label{UoperatorsI2021uop}
The $\mathcal{R}^{\flat}$-module $\cM_{K_\cF^{\flat}}(\Z)$ is free of rank $r_\cF:=|{N}_{1,a_\circ}^{\flat}/N_{1,\cF}^{\flat}|$, with basis $\{v_{w,\cF}^{\flat},\, w\in D_\cF^{\flat}\}$ where $D_\cF^{\flat}\subset {N}_{1,a_\circ}^{\flat}$ is any fixed set of representatives for ${N}_{1,a_\circ}^{\flat}/N_{1,\cF}^{\flat}$. In particular, $\{v_{w,\a}^{\flat},\, w\in W_{a_\circ}\}$ is a canonical basis for the $\mathcal{R}^{\flat}$-module $\cM_{I^{\flat}}(\Z)$, and $v_{1,a_\circ}$ is a canonical generator for the $\mathcal{R}^{\flat}$-module $\cM_{K^{\flat}}(\Z)$.
\end{proposition}
\begin{proof}
This is a particular case of \cite[Proposition \ref{freerank1MK}]{UoperatorsI2021}.
\end{proof}

\subsubsection{Satake homomorphism} 
When $\cF=a_\circ$, Proposition \ref{UoperatorsI2021uop} shows that the $\mathcal{R}^{\flat}$-module $\cM_{K^{\flat}}(\Z)$ is free of rank $1$. 
Accordingly, we have a $\Z$-algebra homomorphism 
$$\dot{\mathcal{S}}_M^{G^{\flat}} \colon (\End_{\Z[G]}(\Z[G/K^{\flat}]))^{\text{opp}}\to \mathcal{R}^{\flat}.$$
This is called the twisted Satake transform and is characterized by
$$(v_{1,K}^{\flat})^h = \dot{\mathcal{S}}_M^{G^{\flat}}(h)\cdot v_{1,K}^{\flat}, \text{ for all }h\in \End_{\Z[G]}(\Z[G/K^{\flat}]).$$ 
This is actually an embedding of $\Z$-algebras and in particular, the algebra $\End_{\Z[G]}(\Z[G/K^{\flat}])$ is a commutative. 
The following theorem gives the "integral" Satake isomorphism:
\begin{theorem}\label{integralsatake}
The twisted Satake homomorphism gives a canonical isomorphism of ${\Z}$-algebras: 
$$\dot{\cS}_M^{G^{\flat}}\colon \End_{\Z[G]}(\Z[G/K^{\flat}]) \iso {\mathcal{R}^{\flat}}^{\dot{W}}.$$
\end{theorem}
\begin{proof}
See \cite[Theorem \ref{Ztwistedsatakeisomorphism}]{UoperatorsI2021}.
\end{proof}
\subsubsection{Bernstein homomorphism}
The following proposition shows that (upon extending the coefficients' ring) the structure of the $(\End_{\Z[G]}(\Z[G/I^{\flat}])])^{\text{opp}}$-module $\mathcal{M}_{I^\flat}(\Z)$ is simple. Set $\Z^q:={\Z[q^{-1}]}$.
\begin{proposition}\label{MiHirk1}
The following homomorphism of right $(\End_{\Z^q[G]}(\Z[G/K^{\flat}])])^{\text{opp}}$-modules
\begin{align*}\hbar_{I^{\flat}}\colon \End_{\Z^q[G]}(\Z^q[G/I^{\flat}])] \longrightarrow \cM_{I^{\flat}}({\Z^q}), \quad h\longmapsto  (v_{1,\a}^{\flat})^h.\end{align*}
is an isomorphism. 
\end{proposition}
\begin{proof}
This is \cite[Theorem \ref{structureMI}]{UoperatorsI2021}.
\end{proof}
In other words, the right $(\End_{\Z^q[G]}(\Z^q[G/I^{\flat}])] ({\Z^q}))^{\text{opp}}$-module $\cM_{I^{\flat}}({\Z^q})$ is free of rank $1$. 
Now, we have by definition an embedding
$$\mathcal{R}^{\flat} \hra  \text{End}_{(\End_{\Z[G]}(\Z[G/I^{\flat}]))^{\text{opp}}}(\cM_{{I^{\flat}}}(\Z)),$$
Accordingly, given Proposition \ref{MiHirk1}, one gets an embedding of $\Z^q$-algebras:
\begin{align*}\dot{\Theta}_{\text{\text{Bern}}}^{\flat}\colon \mathcal{R}^{\flat}\hookrightarrow \End_{\Z^q[G]}(\Z^q[G/I^{\flat}]),\quad m \mapsto \dot{\Theta}_m^{\flat}, \end{align*}
characterized by the property: $m\cdot v_{1,\a}^{\flat}= (v_{1,\a}^{\flat} )^{ \dot{\Theta}_m^{\flat}}$, for any $m\in \mathcal{R}^{\flat}$. 
\begin{lemma}
For any $m\in  \Lambda_M^{\flat}$ and any $m_\circ\in \Lambda_M^{-,\flat}$ such that $m+m_\circ \in \Lambda_M^{-,\flat}$,  
we have
$$\dot{\Theta}_m^{\flat}= i_{m+m_\circ}^{\flat}\circ (i_{m_\circ}^{\flat})^{-1}= (i_{m_\circ}^{\flat})^{-1}  \circ i_{m+m_\circ}^{\flat} \quad \text{ in } \End_{\Z^q[G]}(\Z^q[G/I^{\flat}]).$$
In particular, $\dot{\Theta}_m^{\flat}=i_m^{\flat}$ if $m\in \Lambda_M^{-,\flat}$ and $\dot{\Theta}_m^{\flat}=(i_{-m}^{\flat})^{-1}$ if $m\in \Lambda_M^{+,\flat}$. 
Here, $i_{m}^{\flat}\in \cH_{I^\flat}$ denotes also (by abuse of notation) the corresponding operator in $\End_{\Z^q[G]}(\Z^q[G/I^{\flat}])$. 
\end{lemma}
\begin{proof}
This is \cite[Lemma \ref{explicittheta}]{UoperatorsI2021}. 
\end{proof}
\subsubsection{Compatibility}

\begin{lemma}[Integral compatibility]\label{intcomp}
The Satake and Bernstein twisted $\Z$-homorphisms are compatible, i.e.,
the following diagram (of $\Z$-modules) is commutative:
$$\begin{tikzcd}  \End_{\Z[G]}(\Z[G/K^{\flat}]) \otimes_\Z\Z^q\arrow[hook]{rr}{\dot{\cS}_M^{G^{\flat}} \otimes \text{Id}}&& {\mathcal{R}^{\flat}}^{\dot{W}}\otimes_\Z\Z^q\arrow[hook]{dl}{\dot\Theta_{Bern}^{\flat}}\\
&Z(\End_{\Z^q[G]}(\Z^q[G/I^{\flat}]))\arrow{lu}{\overline{\pi}_{\a}}&
\end{tikzcd}$$
In particular $\pi_\a \circ (\dot\Theta_{Bern}^{\flat})|_{{\mathcal{R}^\flat}^{\dot{W}}}$ is the inverse of $\dot{\cS}_M^{G^{\flat}}$. 
\end{lemma}
\begin{proof}
This is \cite[Lemma \ref{Zcompatibility}]{UoperatorsI2021}. The second claim follows from Theorem \ref{integralsatake}.
\end{proof}
\begin{remark}
The morphism $\overline{\pi}_\a$ in Lemma \ref{intcomp} above, is the restriction of the composition of the following morphisms of $\Z^q$-modules
$$\begin{tikzcd}[row sep= tiny] \End_{\Z^q[G]}(\Z^q[G/I^{\flat}]) \arrow{r}{}[swap]{\simeq}& \Z^q[I^\flat \backslash G/I^\flat]  \arrow{r}{{\pi}_\a}[swap]{\simeq}&\Z^q[I^\flat \backslash G/K^{\flat}],\end{tikzcd}$$
to the center $Z( \End_{\Z^q[G]}(\Z^q[G/I^{\flat}])) $. The right hand side arrow, being the natural extension of the projection map $\pi_\a\colon \Z[G/K_\cF^\flat] \to \Z[G/K^\flat]$. 
\end{remark}
\begin{remark}[Non-commutativity of $\mathbb{U}^\flat$ with $\cH_{K^\flat}(\Z)$]
The non-commutativity of the $\mathbb{U}^\flat$ and the $\flat$-special Hecke algebras can be tracked to the fact that although the natural projection map $\pi_\a \colon \Z^q[G/I^{\flat}] \to \Z^q[G/K^{\flat}]$ yields a morphism of algebras for the center $Z( \End_{\Z^q[G]}(\Z^q[G/I^{\flat}])) \iso \End_{\Z^q[I^\flat]}(\Z^q[G/K^{\flat}])$, 
it only extends to a $\Z^q$-modules morphism $\pi_\a \colon \End_{\Z^q[G]}(\Z^q[G/I^{\flat}]) \to \End_{\Z^q[I^\flat]}(\Z^q[G/K^{\flat}])$. 
%On the level of elements:  If $\overline{h}= \pi_\a(h)$ for any $h \in Z( \End_{\Z^q[G]}(\Z^q[G/I^{\flat}]))$ and $r $ a nontrivial element in $ \mathcal{R}_\cF^{-, \flat}$, then $$\overline{h} \circ \varphi_\a(r)=\overline{\pi}_\a(h) \circ \pi_\a \circ r\circ s_\a \neq  \pi_\a \circ r\circ s_\a\circ \pi_\a(h) = \varphi_\a(r) \circ  \overline{h} .  $$
\end{remark}

\subsection{Integrality}
\begin{lemma}\label{integrality}\label{Zintegrality}
The subring $ {\dot{\Theta}_{\text{\emph{Bern}}}}(\mathcal{R}^{\flat}) \subset \cH_{I^{\flat}}(\Z^q)$ is integral over ${\dot{\Theta}_{\text{\emph{Bern}}}}({{\mathcal{R}^{\flat}}^{\dot{W}}})\subset Z\left(\cH_{I^{\flat}}({\Z^q})\right)$.
\end{lemma}
\begin{proof} It suffices to consider, for any $m \in \mathcal{R}^{\flat}$, the polynomial $P_m(X)=\prod_{w\in W/W_m}(X-\dot{\Theta}_{\text{Bern}}^{\flat}(\dot{w}(m)))$,
where, $W_m$ denotes the isotropy subgroup of $m$ in $W$. 
The polynomial $P_m$ is clearly invariant under the dot-action of $W$, i.e $P_m \in \dot{\Theta}_{\text{Bern}}^{\flat}({\mathcal{R}^{\flat}}^{\dot{W}})[X]$.
%This ends the proof, since by \cite[Corollary \ref{Winvariantcenter}]{UoperatorsI2021} we have $\dot{\Theta}_{\text{Bern}}(\mathcal{R}^{\dot{W}})\subset Z\left(\cH_I({\Z[q^{-1}]})\right)$.
\end{proof}
\begin{remark}Note that by definition $P_m$ is the minimal polynomial annihilating $\dot{\Theta}_m^{\flat}$. 
\end{remark}
\begin{theorem}\label{integralitiyoperatorsu}\label{Zintegralitiyoperatorsutilde}
For any ${u} \in \mathbb{U}_\cF^{\flat}$, there exists a monic polynomial $Q_u(X)=\sum {h}_k X^k \in (\End_{\Z[G]}\Z[G/K^{\flat}])[X]$
$$\sum {h}_k\circ  u^k=0, \text{ in } \text{\emph{End}}_{\Z[P_\cF]}\Z[G/K^{\flat}].$$ 
\end{theorem}

\begin{proof}
Let $r\in \mathcal{R}_{\cF}^{-,\flat}$. By Remark \ref{functo216}, as far as we are concerned with the excursion pairing, it is equivalent to view $r$ in $\mathcal{R}_{\a}^{-,\flat}$. 
Since $ \jmath_\a^\flat ( \mathcal{R}_{\a}^{-,\flat} )=\dot{\Theta}_{\text{\emph{Bern}}}^{\flat}  ( \mathcal{R}_{\a}^{-,\flat} )\subset   {\dot{\Theta}_{\text{\emph{Bern}}}^{\flat}}(\mathcal{R}^{\flat})$, by Lemma \ref{Zintegrality}, there exist a monic polynomial $P_r(X)=\sum_k X^k h_k \in {\dot{\Theta}_{\text{\emph{Bern}}}^{\flat}}({{\mathcal{R}^{\flat}}^{\dot{W}}})[X]$ having $u=\jmath_\a^{\flat}(r)$ as a root (set $h_n ={\bf 1}_{I^\flat}$). 

By Theorem \ref{integralsatake} and Lemma \ref{intcomp}, we have $\bar{h}_k:= \overline{\pi}_\a(h_k)= {h_k}*_{I^\flat}\mathbf{1}_{K^\flat}\in \cH_{K^\flat}(\Z)$ . 
By $B$-equivariance (Corollary \ref{Pinvariance}), it suffices to show that $\sum_k  ({\mathbf{1}_{K^\flat}} \bullet {u^k})*_{K^\flat}\bar{h}_k=0$. 
We have
\begin{align*}
\sum_k  ({\mathbf{1}_{K^\flat}} \bullet {u^k})*_{K^\flat}\bar{h}_k&=\sum_k  ({\mathbf{1}_{K^\flat}} \bullet {u^k})*_{K^\flat}({h_k*_{I^\flat}\mathbf{1}_K})\\
&=[K^\flat:I^\flat]^{-1}\sum_k  ({\mathbf{1}_{K^\flat}} \bullet {u^k})*_{I^\flat}({h_k*_{I^\flat}\mathbf{1}_{K^\flat}}) \\\
&=[K^\flat:I^\flat]^{-1}\sum_k  ( {u^k}*_{I^\flat} \mathbf{1}_{K^\flat}) *_{I^\flat}({h_k*_{I^\flat}\mathbf{1}_{K^\flat}})\\
&=[K^\flat:I^\flat]^{-1}(\sum_k  {u^k} *_{I^\flat} h_k)*_{I^\flat}\mathbf{1}_{K^\flat}*_{I^\flat}\mathbf{1}_{K^\flat}\\
&= P_r(u)*_{I^\flat}\mathbf{1}_K\\
&=0 \in \cC_c(G/K^\flat,\Z).
\end{align*} 
Therefore: $\sum \bar{h}_k\circ u^k =0 \in \End_{\Z[P_\cF]}\Z[G/K^{\flat}],\text{ 
with }\bar{h}_k \in \End_{\Z[G]}\Z[G/K^{\flat}]$.  
\end{proof}

\subsection{Example \texorpdfstring{$\textbf{GL}_2$}{GL2}}

\begin{example}
Let us make some verification by hand for $\G=\textbf{GL}_{2,\Z_p}$. Let $T$ be the diagonal matrices, $B$ the Borel subgroup of upper triangular matrices, $I$ the corresponding Iwahori subgroup $\left\{\begin{pmatrix}*&*\\ p*&*\end{pmatrix}\subset K=\G(\Z_p)\right\}$. 

Set $t_{m,n}:=\begin{pmatrix}p^{m}&\\ &p^{n}\end{pmatrix}$ and $u_\ell:=\begin{pmatrix}1&\ell\\ &1\end{pmatrix}$ for any $\ell,m,n\in \Z_p$. 
We have $It_{1,0}I= \sqcup_{\ell \in \Z_p/ p\Z_p}u_\ell t_{1,0} I$ and so 
$\mathbf{1}_{bK}\bullet \mathbf{1}_{It_{1,0}I} 
=\sum_{a\in \Z_p/ p\Z_p} \mathbf{1}_{bu_\ell t_{1,0} K}$, for any $b\in B$.
Therefore,
\begin{equation*}\label{14}\mathbf{1}_{bK}\bullet \mathbf{1}_{It_{1,0}I}=\mathbf{1}_{bK}*_K T_p-\mathbf{1}_{b t_{0,1} K},\end{equation*} 
where, $T_p=\mathbf{1}_{Kt_{1,0}K}$. 
Accordinly,
\begin{align*}\mathbf{1}_{bK}\bullet \mathbf{1}_{It_{1,0}I}^2
&=(\sum_{\ell\in \Z_p/ p\Z_p} \mathbf{1}_{bu_\ell t_{1,0}K})\bullet\mathbf{1}_{It_{1,0}I}\\
&=\sum_{\ell\in \Z_p/ p\Z_p} \mathbf{1}_{bu_\ell t_{1,0} K}*_K T_p-\mathbf{1}_{b pu_\ell K}\\
&=(\mathbf{1}_{bK}\bullet \mathbf{1}_{It_{1,0}I})*_K T_p-p\mathbf{1}_{bK}*_KS_p%u_\ell \in K
\end{align*}
where, $S_p=\mathbf{1}_{p K}\in \cH_K(\Z)$. 
In conclusion, we see that the $\bullet$-action defines an operator $u\in\End_{\Z[G]}\Z[G/K]$ corresponding to $\mathbf{1}_{It_{1,0}I}\in \cH_I(\Z)$ such that $$u^2- T_p \circ u+p S_p=0,$$
where this time, $T_p,S_p\in \End_{\Z[G]}\Z[G/K]$ are the Hecke operators corresponding to $t_{1,0}$ and $t_{1,1}$, respectively. 
Note that $T_p \circ u \neq u\circ T_p  $ and $u^2- u\circ T_p+p S_p\neq 0$.
\end{example}

\section{A geometric/combinatorial approach}\label{retractionsanduoperators}
The goal of this section is to translate the purely group theoretic ring $\mathbb{U}^{1}$ into a more combinatorial fashion, by describing its induced action on the extended building. 
This will provide a new class of geometric operators on the set of special vertices justifying why $\mathbb{U}^1$-operators may be thought of as a conceptual generalization of the successor operators for trees with a marked end. 

To the fixed maximal split torus $\Sbf$ is attached an apartment $\cA$ in the reduced building $\cB_{\red}=\cB_{\red}(\G,F)$ and an apartment  $\cA_{\ext}=\cA \times V_G$ in the extended building $\cB_{\ext}=\cB_{\ext}(\G,F)$. 
Recall that ${I^1}$ and ${K^1}$ are the $G^1$-fixator of the alcove $\a\subset \cA$ and the special point $a_\circ\in \overline{\a}$, respectively. 
We refer the reader to \cite[\S 2]{UoperatorsI2021} for a reasonably brief exposition of Bruhat--Tits theory and the properties that will be needed in this section.

Since we will be manipulating objects in the reduced and extended buildings, we will be using the subscripts $\square_{\red}$ and $\square_{\ext}$ to distinguish between the two.

\subsection{Preliminaries}
\begin{definition}An element $g \in G$ is said to be strongly type-preserving if: For each triplet $(\cF , \cA_{\ext}^{\prime},w)$, where $\cF$ is a facet contained in an apartment $\cA_{\ext}^{\prime}$ and $w\in W_{\aff}^{\prime}$, for which
 $$\text{\em (i) $g\cdot \cF \subset \cA_{\ext}^{\prime}$ and (ii) $(w \circ g)\cdot \cF=\cF$},$$ then the element $w \circ g$ fixes pointwise the closure $\overline{\cF}$. Set $G_{tp}$ for the set of strongly type-preserving elements in $G$. A subgroup $H\subset G_{tp}$ of strongly type-preserving elements is said to be strongly transitive if it acts transitively on all pairs $(\a'\times V_G,\cA_{\ext}^{\prime})$, where $\cA_{\ext}^{\prime} \subset \cB_{\ext}$ is an apartment and $\a^\prime \times V_G\subset \cA_{\ext}^{\prime}$ is an alcove.
\end{definition}
For more on these types of automorphisms on buildings we refer to \cite{rousseau08}.
\begin{proposition}\label{G1strong}
The subgroup $G_{tp}$ of strongly type preserving elements of $\cB_{\ext}$ is equal to $G_1=\ker \kappa_G$ and its action on $\cB_{\ext}$ is also strongly transitive.
\end{proposition}
\begin{proof}
The subgroup of $G$ consisting of strongly type preserving automorphisms of $\cB_{\ext}$ is the group generated by $\nu_{N}^{-1}(W_{\aff})$ and the root groups $U_\alpha$ for $\alpha \in \Phi$ (see for example \cite[11.10]{rousseau08}), which is equal to the subgroup $G_1= \ker \kappa_G$.  
Finally, $G_1$ is indeed strongly transitive on $\cB_{\ext}$ as shown in \cite[Corollaire 2.2.6]{BT72}.
\end{proof}
\begin{remark}
The action of $G$ on $\cB_{\ext}$ is strongly transitive but non type-preserving in general. One may think of $G_1\backslash \cB_{\ext}$ as the "universal closed alcove" on which $G/G_1\simeq N/N_1$ acts faithfully: it is a (commutative) group of automorphisms of the universal closed alcove, with a translation part, given by $\nu_G$ and a "rotational part", finite, given by the torsion subgroup $G^1/G_1$.
\end{remark} 
\subsection{Iwahori subgroups and retractions}
\begin{lemma}\label{lemretraction}
(i) For any minimal facet $\{a \}\times V_G \in \cB_{\ext}$, there is an apartment $\cA_{\ext,a} \subset \cB_{\ext}$ containing $\a\times V_G$ and $a \times V_G$, along, then, with the whole facet of $a$. 
(ii) There is a unique isometry $\phi\colon \cA_{\ext,a} \to \cA_{\ext}$ fixing $\overline{\a}\times V_G$. 
(iii) The image $\phi(a)$ does not depend on the choice of $\cA_{\ext,a}$. Denote this image by $r_{\cA_{\ext},\a}(a)\nomenclature[E]{$r_{\cA_{\ext},\a}$}{$\colon \cB_{\ext}\to \cA_{\ext}$, Retraction based on the alcove $\a$}$.
\end{lemma}
\proof
\begin{enumerate}
\item[(i)] See \cite[7.4.18]{BT72}.
\item[(ii)] (Existence) By Proposition \ref{G1strong}, there exists $g \in G$ sending the pair $(\cA_{\ext,a}, \a\times V_G)$ to $(\cA_{\ext}, \a\times V_G)$ and fixing pointwise $\overline{\a}\times V_G$.% (existence follows from proposition \ref{buildproperties} (\ref{transitivity})).

(Uniqueness) If $\phi \colon \cA_{\ext,a} \to \cA_{\ext}$ is a second isomorphism, then $\phi\circ  g^{-1}$ is an isometric automorphism of $\cA_{\ext}$ fixing $\overline{\a}\times V_G$, but any automorphism of $\cB_{\ext}$ that fixes a chamber is the identity. Moreover, being an isometry, $\phi$ fixes $\overline{\a}\times V_G$ if and only if it fixes the whole intersection $\cA_{\ext}\cap \cA_{\ext,a}$.
\item[(iii)] Let $\cA_{\ext,a}'$ be another apartment containing $\a\times V_G$ and $a$. Consider the following (possibly non-commutative) diagram 
$$
  \begin{tikzcd}
    \cA_{\ext,a} \arrow{r}{\phi} \arrow[swap]{dr}{\phi''} & \cA_{\ext} \\
     & \cA_{\ext,a}'\arrow[swap]{u}{\phi'}
  \end{tikzcd}
$$
where all maps are (unique) isomorphism fixing $\overline{\a}\times V_G$. 
The uniqueness of (i) implies that $\phi=\phi'\circ \phi'' $, hence, the above diagram is commutative. In addition, the map $\phi''$ being an isometry, must fix $a\in \cA_{\ext,a}\cap \cA_{\ext,a}'$ for all $v \in V_G$. Hence $\phi(a)=\phi'\circ\phi''(a)=\phi'(a)$. 
This proves that $\phi(a)$ does not depend on the chosen $\cA_{\ext,a}$.\qed
\end{enumerate}
\begin{remark}\label{remakrretraction}
By the above lemma, for any point $(a\times v)\subset \cA \times V_G = \cA_{\ext}$ contained in some apartment $\cA_{\ext,a}$ that also contains $\a\times V_G$, there is a unique isomorphism $\phi\colon \cA_{\ext,a} \to \cA_{\ext}$ fixing $\overline{\a}\times V_G$. In particular $\phi$ fixes the component $V_G$, thus 
$$r_{\cA_{\ext},\a}((a,v))=(r_{\cA,\a}(a), v),$$
for some retraction map $r_{\cA,\a}\colon \cB_{\red}\to \cA$ for the reduced building.
\end{remark}
\begin{definition}
The map $r_{\cA_{\ext},\a}\colon\cB_{\ext}\to \cA_{\ext}$ (resp. $r_{\cA,\a}$) defined by lemma \ref{lemretraction} is called the retraction onto $\cA_{\ext}$ based at $\a\times V_G$, (resp. $\cA$ based at $\a$).
\end{definition}
In the following lemma, we interpret the retraction $r_{\cA_{\ext},\a}$ using the Iwahori subgroup.
\begin{lemma}\label{Iwretract} For any $(a,v) \in \cA_{\ext}$,
 $$r_{\cA_{\ext},\a}^{-1}((a,v))={I^1}\cdot (a,v)={I}\cdot (a,v).$$
The fibers of the retraction $r_{\cA_{\ext},\a}$ are exactly the ${I^1}$-orbits on $\cB_{\ext}$.
\end{lemma}
 \begin{remark}\label{remakrretraction2}
Using Remark \ref{remakrretraction}, we see that for any $(a,v) \in \cA_{\ext}$ where $a\in \cA$ and $v \in V_G$, we have $r_{\cA,\a}^{-1}(a)= I\cdot a={I^1}\cdot a$ since $I \subset {I^1} \subset G^1=\ker \nu_G$.
\end{remark}
\begin{proof}
Let $a$ be any point in the apartment $\cA$, $a' \in \cB_{\red}$ such that $r_{\cA,\a}(a')=a$. According to Proposition \ref{G1strong}, the subgroup $G_1$ has a strongly transitive action on $\cB_{\ext}$, thus also on $\cB_{\red}$. Hence there exists $g\in G_1$ that sends the ordered pair $(a', \a)$ to the pair $(a, \a)$. In particular $g \cdot \a=\a$, but since $G_1$ is strongly type-preserving it must fixe pointwise $\a$, hence $g\in I$. 

The isomorphism $\phi_{g^{-1}} \colon g\cdot \cA_{\ext} \to \cA_{\ext} $ fixes the alcove $\a\times V_G$. Thus, by lemma \ref{lemretraction}, we have 
$$(a,v)=(\phi_{g^{-1}}(a'),v)=r_{\cA_{\ext},\a}((a',v)), \quad \forall v\in V_G.$$
Conversely, any $(a',v)=g \cdot (a,v)$ clearly retracts to $(r_{\cA,\a}(a),v)=r_{\cA_{\ext},\a}(a,v)$. Thus $r_{\cA,\a}^{-1}(a)= I\cdot a$ and $r_{\cA_{\ext},\a}^{-1}(a,v)= I\cdot (a,v)$, for any $v\in V_G$.
Finally, the equality $I\cdot (a,v)={I^1}\cdot (a,v)$ follows from ${I^1} = M^1 I$.
 \end{proof}
The following figure describes the case of a tree (e.g., the $\SL_2$ or $\U(3)$ case).  The blue vertices in the closure of a {\color{green}{green}} (resp. a {\color{magenta}{magenta}}) alcove are all in the same $I$-orbit.
\begin{center}
\tikzset{set angles for level/.style={level #1/.append style={sibling angle=60}},
  level/.append style={set angles for level=#1}}
\begin{tikzpicture}[
  grow cyclic,
  level distance=1cm,	
  level/.style={
    level distance/.expanded={\ifnum#1>1 \tikzleveldistance\else\tikzleveldistance\fi},
    nodes/.expanded={\ifodd#1 fill=red!75   \else fill=blue!75\fi}
  },
    level 1/.style={sibling angle=60},
  level 2/.style={sibling angle=60},
  level 3/.style={sibling angle=30},
  level 4/.style={sibling angle=25},
  nodes={circle,draw,inner sep=+0.9pt, minimum size=2pt},]
\path[rotate=90]
  node (A) [draw,top color=red!90, bottom color=blue!10,minimum size=4pt]  [above right]{}
child[]{
    node [above,circle,draw=black,top color=red!10,bottom color=blue!90,minimum size=4pt] {} 
    child[black] foreach \cntII in {1,...,2} { 
      node  [draw=black,top color=red!90,bottom color=blue!10,minimum size=4pt]  {}
      child[green] foreach \cntIII in {1,...,2} {
        node  [draw=black,top color=red!10,bottom color=blue!90,minimum size=4pt] {} 
        child[black] foreach \cntIV in {1,...,2} {
          node  [draw=black,top color=red!90,bottom color=blue!10,minimum size=4pt]  {}
          child[magenta] foreach \cntIV in {1,...,2} {
          node  [draw=black,top color=red!10,bottom color=blue!90,minimum size=4pt]  {} }}}}};
\node[draw=none] at (0.2,0.5) {$\a$};
\node[draw=none] at (0.35,1.1) {$a_\circ $};
    %\node[draw=white] at (3,-3) {\(v_1^t\)};
    %\node[draw=white] at (3,4.5) {\(\)};
  %\node[draw=white] at (0,5*1.15) {1};
%\draw[blue] (0:0)-- (60:1);
\end{tikzpicture}     
\end{center}
\subsection{Geometric \texorpdfstring{$\mathbb{U}$}{U}-operators}\label{ALgeo}
Set $\cA^\circ_{\ext}:=M\cdot (a_\circ,0_{V_G}) =\left\{a_{m}:=\left(a_\circ + \nu_M(m),\nu_G(m)\right) , \forall m\in M \right\} \subset \cA_{\ext} \text{ and }\cB_{\ext}^{\circ}:=G\cdot (a_\circ,0_{V_G})$.

We define an "excursion pairing" 
$$\begin{tikzcd}[column sep= tiny]
M\arrow{rrr} \times \Z[\cA^\circ_{\ext}]   &&& \Z[\cB^\circ_{\ext}],&
(m,a_{m'}  ) \arrow[rrr, mapsto]&&& \cU_m a_{m'}
\end{tikzcd}$$
where $\cU_ma_{m'}\nomenclature[e]{$\cU_m$}{The geometric $\mathbb{U}$-operator corresponding to $m\in M^-$}$ is defined to be the formal sum of vertices appearing in the fiber $$r_{\cA_{\ext},m' \cdot \a}^{-1}(m\cdot a_{m'})= r_{\cA_{\ext},m' \cdot \a}^{-1}(a_{mm'}).$$
\begin{lemma}\label{equivariant}
Let $m\in M^-$ and $m' \in M$. We have
$$\cU_{m} a_{m'}= m' \cdot \cU_{m} a_1.$$
In other words, the operator $\cU_{m}$ is $M$-invariant and in particular; the restriction of the above pairing to the semigroup $M^-$ defines a left action \begin{tikzcd}[column sep= tiny]
M^-\arrow{rrr} \times \Z[\cA^\circ_{\ext}]  \arrow{rrr}&&& \Z[\cB^\circ_{\ext}]
.\end{tikzcd}
\end{lemma}
\proof 
By Lemma \ref{Iwretract}, we have
\begin{align*}\cU_ma_{m'}&=I_{m'\cdot \a} \cdot a_{mm'}\\
&= m'Im'^{-1} m m' \cdot (a_{\circ},0_{V_G})\\
\overset{\text{(1)}}&{=}m'Im \cdot (a_{\circ},v)\\
\overset{\text{(2)}}&{=}m'I^+m \cdot (a_{\circ},0_{V_G})&\text{\cite[Lemma \ref{normalizingI}]{UoperatorsI2021}}
,\end{align*}
%here we abuse notation and will use the same letter for a class in $\Lambda_M$ and a representative for it in $M$. 
For (1), although $m,m'$ may not commute as elements of $M$, they do commute modulo $M_1\subset I\subset {I^1}$, which shows $m'^{-1}mm'K=mK$. The line (2) requires $m \in M^-$, which guarantees $I^+ \cap m {K^1} m^{-1} = m I^+ m^{-1}$. Hence,
%$$\cU_ma_{m'}= \sum_{ i \in m'Im'^{-1}/(m'Im'^{-1}\cap m K m^{-1})} i \cdot a_{m+m'}.$$
$$\cU_ma_{m'}=m' \sum_{ i \in I^+/m I^+ m^{-1}} i \cdot a_{m}.$$
Note that the summation above is by definition $\cU_{m} a_1$.\qed 

Now, we extend the action of $\cU_{m}$ to the set $\cB_{\ext}^{\circ}=B\cdot (a_\circ,0_{V_G})$, 
this equality follows from $G = BK$ (Theorem \ref{Iwasawadec}).
\begin{lemma}\label{funddomainU}
For each $g \in G$, set $a_g:=g \cdot (a_\circ,0_{V_G})\nomenclature[E]{$a_g$}{$g \cdot (a_\circ,0_{V_G})\in \cB_{\ext}^{\circ}$ for $g\in G$}$\footnote{This is compatible with our previous notation $a_m=m \cdot (a_\circ,0_{V_G})$ for $m \in M$.}. The intersection
$$  \cA_{\ext}^\circ \cap U^+ \cdot a_g $$
consists of a single\footnote{Actually, the extended apartment $\cA_{\ext}$ is a fundamental domain for the action of $U^+$ on $\cB_{\ext}$.}
 vertex $\{a_{m_g}\}$ for some $m_g \in M$, unique modulo $M^1$. 
 \end{lemma}
\begin{proof}The intersection is nonempty by the decomposition $G=U^+ M K$. Suppose that there exist two $m,m' \in M$ such that $U^+ a_m=U^+ a_{m'}$. Let $u\in U^+$ such that $u \cdot a_m=a_{m'}$, i.e. $u \cdot a_m = a_{m'}=m'm^{-1} \cdot a_m $, hence $m m'^{-1} u \in P_{a_m} $, using the decomposition (see \cite[Proposition \ref{propertiesII3.3} and Proposition \ref{buildproperties}]{UoperatorsI2021})
$$P_{\{m\cdot a_\circ\}}=N_{\{m\cdot a_\circ\}} U_{\{m\cdot a_\circ\}}=N_{\{m\cdot a_\circ\}} U_{\{m\cdot a_\circ\}}^- U_{\{m\cdot a_\circ\}}^+.$$
Write $(m m'^{-1}) u= n u_- u_+$ for some $n \in N_{\{m\cdot a_\circ\}}$ and $u_\pm \in U_{\{m\cdot a_\circ\}}^\pm$. But $N\cap U^+U^-=\{1\}$ \cite[5.15]{BR65}, hence 
$$m m'^{-1}=n \in N_{\{m\cdot a_\circ\}}\text{ and }u= u_+ \in  U_{\{m\cdot a_\circ\}}^+,$$
and so $a_{m'}=u a_m =a_m$. This proves the lemma.
	
An alternative argument: For each $u\in U^+$, the restriction of the action of $u^{-1}$ to $u\cdot \cA_{\ext}$ is the unique isomorphism (uniqueness resembles the proof of Lemma \ref{lemretraction}) mapping $u\cdot \cA_{\ext}$ to $\cA_{\ext}$ and fixing their intersection pointwise. In particular, it fixes $a_{m'}=u \cdot a_m \in u\cdot \cA_{\ext}\cap \cA_{\ext}$, thus $a_m=a_{m'}$.\end{proof}
\begin{lemma}\label{equivariantB}Let $m\in M^-$, $g\in G$ and $u_g\in U^+$ verifying $a_g = u_g a_{m_g}$. Define
$$\cU_{m} a_g:= u_g \cdot \cU_{m} a_{m_g}\in \Z[\cB_{\ext}^\circ].$$
This definition does not depend on the chosen $u_g \in U^+$.%, in other words $\cU_{m}$ is $U^+$-equivariant.
\end{lemma}
\begin{proof}
By Lemma \ref{equivariant}, we have $\cU_{m} a_g= u_gm_g \cdot \cU_{m} a_1$. %=m_g (\underbrace{m_g^{-1}u_gm_g}_{\in U^+}) \cdot \cU_{m} a_\circ.$$it suffices then to suppose that $a_m=a_\circ$. 
Let $u\in U^+$ be another element such that $a_g= um_g \cdot (a_{\circ},0_{V_G})$. Therefore, 
by \cite[Proposition \ref{propertiesII3.3} and \ref{buildproperties}]{UoperatorsI2021}
$$m_g^{-1}u_g^{-1} um_g \in P_{\{a_\circ\}}^+\cap G^1\cap U^+=U_{a_\circ}^+=I^+.$$ 
But $\cU_{m} (a_\circ,0_{V_G})= I \cdot a_m$, hence $\cU_{m} a_g = u_gm_g \cdot \cU_{m} (a_\circ,0_{V_G})=um_g \cdot \cU_{m} (a_\circ,0_{V_G})$.
\end{proof}
\begin{corollary}\label{Bequivariance}
For any $m\in M^-$, we have
$$\cU_m\in \text{\emph{End}}_{\Z[B]}\Z[\cB_{\ext}^{\circ}].$$
%compare with Corollary \ref{injectU}.
\end{corollary}
\begin{proof}
This is a straightforward consequence of Lemmas \ref{equivariant} and \ref{equivariantB}.\qedhere
\end{proof}
\begin{remark}
By definition, the map
$M^- \longrightarrow \text{End}_{\Z[B]}\Z[\cB_{\ext}^{\circ}],\quad m \longmapsto \mathcal{U}_m$
factorizes through the quotient ${\Lambda}_M^{-,1}:= M^-/M^1\subset  \Lambda_M/(\Lambda_{M})_{\text{tor}}$.\end{remark}
\begin{lemma}\label{additivitysuccesors}
For any $m,m'\in {\Lambda}_M^{-,1}$, we have
$$\cU_m \circ \,\cU_{m'}=\cU_{m+m'}.$$
In particular, the operators $\cU_ m$ and $ \cU_{m'}$ commute.
\end{lemma}
\begin{proof}Here we will abuse notation and use the same letter for a class in ${\Lambda}_M^{-,1}$ and a representative for it in $M$. 

Recall that the Iwasawa decomposition $G=BK$ yields $\cB_{\ext}^{\circ}= B \cdot (a_\circ,0_{V_G})$. Thus, by the $B$-equivariance of the operators $\cU_m$ and $\cU_{m'}$, it suffices to verify $\cU_m \circ \cU_{m'} a_{1}=\cU_{m+m'} a_{1}$ for $m,m' \in  {\Lambda}_M^{-,1}$. 
We thus have
\begin{align*}\cU_{m}\circ \cU_{m'} a_{1}&=\cU_{m} \sum_{ i' \in I^+/m' I^+ m'^{-1}} i'm' \cdot (a_{\circ},0_{V_G})\\
&= \sum_{ i' \in I^+/m' I^+ m'^{-1}} i'm'\, \cU_{m}  \cdot (a_{\circ},0_{V_G})&(B\text{-equivariance})\\
&=\sum_{ i' \in I^+/m' I^+ m'^{-1}}\sum_{ i \in I^+/ m I^+ m^{-1}} i'm'im  \cdot (a_{\circ},0_{V_G})
\end{align*}
Since $m, m'$ are chosen in $M^-$, we have
$$m'm I^+ m^{-1} m'^{-1} \subset m I^+ m^{-1} \subset I^+$$
and if $\mathscr{R}, \mathscr{R}' \subset I^+$ is a set of representatives for $I^+/m I^+m^{-1}$ resp. $I^+/m' I^+m'^{-1}$ , then
$$\mathscr{R}'' = \{i'm'im'^{-1} \colon i'\in \mathscr{R}' \text{ and }  i \in \mathscr{R} \}$$
is a set of representatives for $I^+/m'mI^+(m'm)^{-1}$. Therefore,
\begin{align*}\cU_{m}\circ \cU_{m'} a_1=\sum_{ i'' \in \mathscr{R}''} i'' \cdot a_{m+m'}&=\cU_{m+m'} a_1. \qedhere\end{align*}
\end{proof}
\begin{remark}Note that the definition of the above operators $\{\mathcal{U}_m \colon m\in {\Lambda}_M^{-,1}\}$ is independent from the alcove $\a$ and the special point $a_\circ$. 
Actually, such operators should be imagined (could be defined in the first place) as "successor" operators with respect to the point corresponding to $B$ in the building at infinity.\end{remark}
\begin{definition}
Define $\mathcal{U}\subset\text{\emph{End}}_{\Z[B]}\Z[\cB_{\ext}^{\circ}]$ to be the ring generated by the operators $\{\cU_{m}\colon m\in {\Lambda}_M^{-,1}\}$.
\end{definition}
\begin{corollary}\label{isomlambdageo}
The following map is a natural isomorphism of rings 
$$\begin{tikzcd}  \mathcal{R}_{\a}^{-,1}=\Z[{\Lambda}_M^{-,1}]\arrow{r}{\simeq}& \mathcal{U},& m \arrow[r, mapsto]& \mathcal{U}_m.\end{tikzcd}$$
\end{corollary}
\proof
We claim that the kernel of the map $M^- \longrightarrow \mathcal{U}$ given by $m\mapsto \cU_m$ is precisely $M^1$. {Indeed, we have:
\begin{enumerate}[nosep]
\item The only points of the apartment where the alcove based retraction $r_{\cA_{\ext},\a}$ has trivial fiber (i.e. reduced to the point itself) is the closure of the alcove $\a$ and this is true by Lemma \ref{Iwretract}.
\item The closure of the alcove $\a \times V_G$ contains a single point of $M^-\cdot  (a_\circ,0_{V_G})$, namely $(a_0,0_{V_G})$ itself. Indeed, as in \cite[Example \ref{exfomega}]{UoperatorsI2021} we know that 
$$\overline{\a}= \{a\in \A_{\red}(\G,\Sbf)\colon 0\le  \alpha(a-a_\circ) \le n_{\alpha}^{-1}  \text{ for all } \alpha \in \Phi_{\red}^+\}.$$
But $M^-=\{m \in M\colon \alpha(\nu(m)+a_\circ) \le 0 , \,\, \forall \alpha \in \Phi_{\red}^+\}$, thus $\overline{\a} \cap (M^- \cdot  (a_\circ,0_{V_G}))=\{  (a_\circ,0_{V_G})\}$.
\end{enumerate}}
In conclusion, the above claim together with Lemma \ref{additivitysuccesors} show the corollary.\qed
\begin{lemma}\label{geometricfaithfulness}The geometric action of $\mathcal{U}$ on $\Z[\cB_{\ext}^\circ]$ is faithful.
\end{lemma}
\proof This is morally similar to the proof of Lemma \ref{faithfulness}. 
Let $m_1, \cdots, m_r\in {\Lambda}_M^{-,1}$ distinct and $s_1,\cdots, s_r\in \Z$. We then have
${K^1} \cdot \sum_is_i \, \mathcal{U}_{m_i} a_1%= \sum_is_i \, {K^1}\cdot I \cdot a_{m_i} 
 =  \sum_is_i \,  ({K^1}m_i{K^1}) \cdot a_1
$. 
Hence, if ${K^1} \cdot \sum_is_i \, \mathcal{U}_m a_1=a_1$, then by Cartan decomposition (\cite[Proposition \ref{cartanK}]{UoperatorsI2021}), one must have $s_1=1$ if $m_i =1 \in \Lambda_M^-$ and $s_i=0$ otherwise. \qed

Combining Lemmas \ref{Bequivariance}, \ref{additivitysuccesors} and \ref{geometricfaithfulness} we get a natural embedding of rings 
\begin{equation}\label{Uembddinggeometric}\begin{tikzcd}\mathcal{U} \arrow[r,hook] & \End_{\Z[B]}\Z[\cB_{\ext}^\circ] %\arrow[r] & \End_{R[B]}R[\cB_{\ext}^\circ]
.\end{tikzcd}\tag{$\dag$}\end{equation}
\begin{theorem}\label{U=A}
The subring $\mathcal{U}\subset \End_{\Z[B]}\Z[\cB_{\ext}^\circ]$ is integral over $\End_{{\Z}[G]}\Z[\cB_{\ext}^\circ]$.
\end{theorem}
\begin{proof}
We have by definition $G/{K^1}\simeq \cB_{\ext}^\circ $. %$\simeq \cB_{\red}^\circ\times \Z^{\text{rk}(\mathbf{Z}_{c,sp})}$ 

Thanks to Corollary \ref{isomlambdageo}, we have a canonical isomorphism of rings
$$\mathbb{U}^1 \iso \mathcal{U}$$
given on basis elements by ${h_{m,\a}^1} \mapsto \mathcal{U}_m$ for all $m\in {\Lambda}_M^{-,1}$. 
This identification between ${\mathbb{U}}^1$ and $\mathcal{U}$ is compatible with their respective actions on $\Z[G/{K^1}]$ and $\Z[\cB_{\ext}^\circ]$. 
Indeed,  for any $m\in {\Lambda}_M^{-,1}$ we have $$\cU_{m} a_\circ=\sum_{ i \in {I^1}^+/m {I^1}^+ m^{-1}} i \cdot a_{m} \text{ and }{\bf 1}_{{K^1}}\bullet {h_{m,\a}^1}= \sum_{ i \in {I^1}^+/m {I^1}^+ m^{-1}} {\bf 1}_{i m {K^1}}.$$
Therefore, the $G$-equivariant isomorphism $\Z[G/{K^1}] \to \Z[\cB^\circ_{\ext}]$ yields an isomorphism of rings 
$$\End_{\Z[B]}\Z[G/K^1] \iso \End_{\Z[B]}\Z[\cB^\circ_{\ext}],$$
which identifies $\mathbb{U}^1$ with $\mathcal{U}$ and $\End_{\Z[G]}\Z[G/K^1]$ with $\End_{\Z[G]}\Z[\cB^\circ_{\ext}]$. 
Accordingly, the integrality of ${\mathbb{U}^1}$ over $\End_{\Z[G]}\Z[G/K^1]$ (Theorem \ref{Zintegralitiyoperatorsutilde}) is equivalent to the integrality of $\mathcal{U}$ over $\End_{\Z[G]}\Z[\cB_{\ext}^\circ]$.
\end{proof}
\section{Filtrations and \texorpdfstring{$\mathcal{U}$}{U}-operators}\label{filtrationandU}
In this subsection we will present yet another alternative point of view for the geometric operators ring $\mathcal{U}$. The main goal here is Theorem \ref{Uopfiltration}, which answers a question suggested by C. Cornut.
\subsection{Definition}
%Let $\Gamma$ be the poset $(\R,+,\le)$. 
Set $\mathbf{F}=\mathbf{F}(G)$ to be the set of all $\R$-filtrations on $G=\G(F)$ \cite[page 77]{Cornut17}, this is the vectorial Tits building $\mathbf{F}$ \cite[Chapter 4]{Cornut17}. By \cite[\S 6.2]{Cornut17}, the extended Bruhat--Tits building $\mathcal{B}_{\ext}=\mathcal{B}(\G,F)_{\ext}$ is an affine $\mathbf{F}$-space \cite[\S 5.2.1.]{Cornut17}, in particular there is a $G$-equivariant right "action" of the vectorial Tits building $\mathbf{F}$ on $\mathcal{B}_{\ext}$,
$$\begin{tikzcd}[column sep= tiny]\mathcal{B}_{\ext}& \times & \mathbf{F} \arrow{rrr}&&& \mathcal{B}_{\ext} &&&(a,\mathcal{F})\arrow[rrr, mapsto] &&& a+\mathcal{F}.\end{tikzcd}$$
Define $\alpha_{\mathcal{F}}(a) = a+\mathcal{F}$ and let $\mathbf{F}^{\circ}$ be the subset of all filtrations $\mathcal{F} \in \mathbf{F}$ such that $\alpha_{\mathcal{F}}(\mathcal{B}_{\ext}^\circ)=\mathcal{B}_{\ext}^\circ$. Thus we get a right "action":
$$\begin{tikzcd}[column sep= tiny]\mathcal{B}_{\ext}^\circ& \times& \mathbf{F}^{\circ} \arrow{rrr}&&& \mathcal{B}_{\ext}^\circ &&&(a,\mathcal{F})\arrow[rrr, mapsto] &&&\alpha_{\mathcal{F}}(a)=a+ \mathcal{F}.\end{tikzcd}$$
\begin{example}\text{\em \cite[\S 6.1]{Cornut17}} Let $\cV\neq 0$ be a $F$-vector space of dimension $n \in \N$ and $\G=\GL(\cV)$. Set
$$\Sbf(\cV):=\{\mathcal{S} \subset \mathbb{P}^1(\cV)(F)\colon \cV=\oplus_{L\in \mathcal{S}}L\},$$
$$ \mathbf{F}(\cV):=\{\text{$\R$-filtrations on $V$}\}.$$
Since $F$ is locally compact, any $F$-norm on $\cV$ is splittable by $\mathcal{S} \in \Sbf(\cV)$, that is
$$\forall v \in \cV \colon \quad  \alpha(v)=\max\{\alpha(v_L)\colon L\in \mathcal{S}\},\text{ where }v=\sum_{L\in \mathcal{S}}v_L, v_L\in L.$$
The extended building $\cB(\GL(\cV),F)_{\ext}$ identifies naturally with $\cB(\cV)$; the set of all splittable (by some $\mathcal{S} \in \Sbf(\cV)$) $F$-norms on $\cV$. 
The group $G$ acts on $\cB(\cV)$ by $(g\cdot \alpha)(v)=\alpha(g^{-1}v)$, for any $g\in G$ and any splittable $F$-norm $\alpha$ on $\cV$. 
The action of $\mathbf{F}(G)\simeq\mathbf{F}(\cV)$ on $\cB(\cV)$ is described as follows: For any $\cF \in \mathbf{F}(\cV)$ and $\alpha \in \cB(\cV)$, there exists $\mathcal{S} \in \Sbf(\cV)$ such that $\alpha$ is splittable by $\mathcal{S}$ and $\cF\in  \mathbf{F}({\mathcal{S}})$, the action is then given by 
$$(\alpha+\cF)(v)=\max\left  \{ q^{-\cF(L)} \alpha(v_{L})\colon L\in \mathcal{S} \right\}.$$
\end{example}
\begin{definition}
Define 
$$\beta_{\mathcal{F}}(b)=\sum_{\alpha_{\mathcal{F}}(a)=b}a.$$\end{definition}
Let $P_{\mathcal{F}}$ be the stabilizer of $\mathcal{F}$ in $G$, this is the group of $F$-points of a parabolic subgroup of $\G$ \cite[\S 2.2.8]{Cornut17}. Then by $G$-equivariance we see that
$$\alpha_{\mathcal{F}},\beta_{\mathcal{F}}\in \text{{End}}_{P_{\mathcal{F}}}\left(\Z[\mathcal{B}_{\ext}^\circ]\right).$$
Set $\mathbf{F}^{\circ}(B)=\{\mathcal{F}\in \mathbf{F}^{\circ}\colon B\subset P_{\mathcal{F}}\}$, then
\begin{lemma}For any $\mathcal{F,G}$ in $\mathbf{F}^{\circ}(B)$, we have
$$\alpha_{\mathcal{F}}\circ \alpha_{\mathcal{G}}=\alpha_{\mathcal{G}} \circ \alpha_{\mathcal{F}} \text{ and }\beta_{\mathcal{F}}\circ \beta_{\mathcal{G}}=\beta_{\mathcal{G}} \circ \beta_{\mathcal{F}} \text{ in }\End_{\Z[B]}\Z[\mathcal{B}_{\ext}^\circ].$$
\end{lemma}
\proof The first commutativity rule results from the condition AC \cite[\S 5.2.5]{Cornut17}, this condition ensures that for any $a \in \cB_{\ext}^\circ$
$$\alpha_{\mathcal{G}}\circ \alpha_{\mathcal{F}}(a)= (a+\cF)+ \mathcal{G}= a+(\cF+ \mathcal{G})=(a+\mathcal{G})+ \cF=\alpha_{\mathcal{F}}\circ \alpha_{\mathcal{G}}(a).$$
The same condition {\em loc. cit} also yields\begin{align*}\beta_{\mathcal{F}}\circ \beta_{\mathcal{G}}(b)&=\beta_{\mathcal{F}} \left(\sum_{\alpha_{\mathcal{G}}(a)=b}a\right)\\
&=\sum_{\alpha_{\mathcal{G}}(a)=b}\sum_{\alpha_{\mathcal{F}}(c)=a}c\\
&=\sum_{\alpha_{\mathcal{G}+ \cF}(c)=b} c\\
&=\sum_{\alpha_{ \cF+\mathcal{G}}(c)=b} c=\beta_{\mathcal{G}}\circ \beta_{\mathcal{F}}(b).\qed 
\end{align*}
\subsection{Formula}
\begin{lemma}\label{formula}
Let $\mathcal{F} \in \mathbf{F}^\circ$, $b \in \mathcal{B}_{\ext}^\circ$ and fix $ a \in \mathcal{B}_{\ext}^\circ$ such that $\alpha_{\mathcal{F}}(a)=b$, then
$$\beta_{\mathcal{F}}(b)=\sum_{h} h a, \quad h\in U_{\mathcal{F}} \cap \text{\em Stab}(b)/ U_{\mathcal{F}}\cap \text{\em Stab}(a).$$
\end{lemma}
\begin{proof}
For $x\in \cB_{\ext}$ and $\mathcal{F} \in \mathbf{F}^\circ$, we may consider the geodesic ray \cite[\S 5.2.11]{Cornut17}
$$\R_+ \longrightarrow \cB_{\ext}, \quad t \longmapsto x + t \mathcal{F}.$$
Since we are working over a complete field, any geodesic ray is standard meaning it is contained in some apartment \cite[6.2.8]{Cornut17}. Thus, for any $a'\in \alpha_{\mathcal{F}}^{-1}(b)$, there are apartments $\cA$ and $\cA'$ containing respectively $\{a+t \mathcal{F}, t\ge 0\}$ and $\{a' + t \mathcal{F}, t\ge 0\}$.

Suppose that $a+\mathcal{F} = a'+\mathcal{F} = b\in \cA \cap \cA'$ and extend the half lines $a+t\mathcal{F}$ and $a'+t\mathcal{F}$ to geodesic lines $L$ and $L'$, given by $b+t\mathcal G$ and $b+t\mathcal G'$, for $\mathcal G$ and $\mathcal G'$ opposed to $\mathcal{F}$. There exists a $u$ in $U_{\mathcal{F}}$ (the unipotent radical of $P_{\mathcal{F}}$) mapping $\mathcal G$ to $\mathcal G'$. It also fixes $a+t\mathcal F = b+(t-1)\mathcal F = a'+t\mathcal F$ for $t\gg 0$. It follows that $u$ maps $L$ to $L'$ and fixes their intersection\footnote{One may be surprised by the fact that any $u$ mapping $\mathcal G$ to $\mathcal G'$ fixes $b$. This is because $\mathcal G$ and $\mathcal G'$ are not merely "opposed to $\mathcal{F}$" in the building at infinity, they are actually and by construction, "opposed to $\mathcal{F}$ at $b$", i.e. $b+t\mathcal G$ and $b+t\mathcal{F}$ ($t\ge  0$) make a $\pi$-angle at $b$.}, which contains $b + t\mathcal{F}$ for all $t\ge 0$. It then maps $a=b+\mathcal G$ to $a'=b+\mathcal G'$. This proves the lemma.\end{proof}
\subsection{Comparison with  \texorpdfstring{$\mathcal{U}$}{U}}
The extended standard apartment $\A(\G,\Sbf)_{\ext}$ attached to $\Sbf$ is an affine space over the real vector space $X_*(\Sbf)\otimes_\Z \R$ and $m\in M$ acts on it by translation $\nu_M(m)$ (See \cite[\S \ref{emptyaprtment}]{UoperatorsI2021}. Likewise, the vector space $X_*(\Sbf)\otimes_\Z \R$ is canonically isomorphic to the apartment $\mathbf{F}(S)$ of $\mathbf{F}$ corresponding to $\Sbf$ {\cite[Definition 9 and \S 4.1.13 ]{Cornut17}}. 
Using this identification, for each $m\in M^-$ set $\mathcal{F}_m\in \mathbf{F}(S)\subset \mathbf{F}$ for the filtration corresponding
\footnote{{We want $m\cdot  a_1 + \mathcal F_m = a_{1}$, i.e. $a_{1} + \nu(m)  + \mathcal F_m = a_{1}=(a_\circ,0_{V_G})$ in the apartment $\cA_{\ext}=\cA_{\ext}(\Sbf)$, so we want $\nu(m) + \mathcal F_m = 0$, i.e. $\mathcal F_m$ is opposed to $\nu(m)$. Alternatively: $\mathcal F_m$ corresponds to $m^{-1}$, which is in $M^+$.}} to $-\nu_M(m) \in X_*(\Sbf)\otimes_\Z \R$. 
Then
$$\alpha_{{\mathcal{F}_m}}, \beta_{{\mathcal{F}_m}} \in \End_{\Z[B]} \Z[\mathcal{B}_{\ext}^\circ].$$
\begin{theorem}\label{Uopfiltration}Let $m \in M^-$, then we have
$$\beta_{{\mathcal{F}_m}}=\mathcal{U}_m\text{ in }\End_{\Z[B]}\Z[\mathcal{B}_{\ext}^\circ].$$
\end{theorem}
\proof
Let ${K^1}$ be the stabilizer of the distinguished base point $(a_\circ,0_{V_G})\in \mathcal{B}_{\ext}^\circ$. This is the open compact subgroup
$${K^1}=P_{a_\circ}\cap \ker \nu_{G}=P_{a_\circ}\cap G^1.$$
It is also the group of elements $g\in P_{a_\circ}$ such that $ \kappa_G(g)$ is torsion, hence $K$ has finite index in ${K^1}$.

By the formula of Lemma \ref{formula} and Lemma \ref{bijection42}, we have
$$\beta_{{\mathcal{F}_m}}(a_1)=\sum_{h\in \in U_{{\mathcal{F}_m}} \cap {K^1}/ U_{{\mathcal{F}_m}}\cap m{K^1}m^{-1}} h m a_1=\sum_{h\in I^+/ mI^+m^{-1}} h m a_\circ.$$
Therefore, this shows by $B$-equivariance that $\beta_{{\mathcal{F}_m}}=\mathcal{U}_m$. \qed 

\bibliographystyle{amsalpha}
\bibliography{Reda_library2021}

\end{document}

%% file: settings_custom.tex
\usepackage[bottom]{footmisc}
\usepackage{tikz}
\usepackage{tikz-cd}
\usepackage{mathrsfs}
\usepackage{supertabular}
\usepackage[shortlabels]{enumitem}
\usetikzlibrary{lindenmayersystems}
\newcommand{\Gal}{Gal}
\newcommand{\GL}{GL} 
\newcommand{\SL}{SL}

\newcommand{\surjto}{\twoheadrightarrow}

\newcommand{\Spec}{Spec} % spectrum
\newcommand{\Tr}{Tr} % trace
\newcommand{\ext}{ext} % extension
\newcommand{\Hom}{Hom} % homomorphism
\newcommand{\End}{End} % endomorphism

\newcommand{\Aff}{\textit{\emph{\text{Aff}}}}

\newcommand{\B}{\mathbf{B}}

\renewcommand{\a}{\mathfrak{a}}
\newcommand{\p}{\mathfrak{p}}

\newcommand{\cA}{\mathcal{A}}
\newcommand{\cF}{\mathcal{F}}
\newcommand{\cG}{\mathcal{G}}

\newcommand{\cV}{\mathcal{V}}

\newcommand{\cB}{\mathcal{B}}

\newcommand{\hra}{\hookrightarrow}

\newcommand{\cP}{\mathcal{P}}
\newcommand{\cS}{\mathcal{S}}

\renewcommand{\P}{\mathbf{P}}

\newcommand{\cT}{\mathcal{T}}

\newcommand{\cU}{\mathcal{U}}
\DeclareFontEncoding{OT2}{}{} % to enable usage of cyrillic fonts
  
\newcommand{\red}{red}

\newcommand{\aff}{\textit{\emph{\text{aff}}}}

\newcommand{\Fr}{Fr}

\newcommand{\Lbf}{\mathbf{L}}
\newcommand{\Nbf}{\mathbf{N}}

\newcommand{\Mbf}{\mathbf{M}}
\newcommand{\Sbf}{\mathbf{S}}

\newcommand{\Res}{\text{Res}}

\newcommand{\Sh}{Sh}

\newcommand{\rec}{rec}

\newcommand*\iso{
  \xrightarrow{\raisebox{-0.25 em}{\smash{\ensuremath{\sim}}}}%
}
\usepackage{dsfont}
\usepackage{upgreek}
\usepackage{stmaryrd}
\usepackage{color}
\usepackage{amsthm,amssymb,amsmath}
\usepackage{lastpage}
\usepackage{centernot}
\usepackage{graphicx}
\usetikzlibrary{cd}
\usetikzlibrary{trees}
\usetikzlibrary{fit}

\tikzset{%
  highlight/.style={rectangle,rounded corners,fill=blue!15,draw,fill opacity=0.5,thick,inner sep=0pt}
}

\usepackage{nicematrix}
\usepackage{url}
\usetikzlibrary{hobby,backgrounds,calc,trees}
\usetikzlibrary{matrix,arrows,decorations.pathmorphing}

\def\1{\mathds{1}}

\def\cC{\mathcal{C}}

\def\cO{\mathcal{O}}
\def\cH{\mathcal{H}}
\def\cX{\mathcal{X}}

\def\cZ{\mathcal{Z}}

\def\cM{\mathcal{M}}

\def\A{\mathds{A}}

\def\Z{\mathds{Z}}
\def\N{\mathds{N}}
\def\Q{\mathds{Q}}
\def\R{\mathds{R}}
\def\C{\mathds{C}}

\def\T{\mathbf{T}}

\def\U{\mathbf{U}}
\def\G{\mathbf{G}}

\def\Hom{\text{Hom}}

\def\Spec{\text{Spec}}
\def\Gal{\text{Gal}}

\newcommand{\authnote}[2][]{\noindent {\if!#1!  {\bf TODO} \else {\small \bf #1} \fi: #2}}

\newcounter{tasknumber}[subsection]
\newcommand{\task}[2][]{%
  \addtocounter{tasknumber}{1}%
  \begin{center}%
  \framebox[1.1\width]{\begin{minipage}{0.9\textwidth}%
  \textbf{Task \arabic{tasknumber}} \textit{\if!#1(unassigned)!\else (#1)\fi}: {#2}%
  \end{minipage}}%
  \end{center}%
}

\newcounter{assumptionnumber}
\newcommand{\assumption}[2][]{%
  \addtocounter{assumptionnumber}{1}%
  \begin{center}%
  \framebox[1.1\width]{\begin{minipage}{0.9\textwidth}%
  \textbf{Assumption \arabic{assumptionnumber}} \textit{\if!#1!\else (#1)\fi}: {#2}%
  \end{minipage}}%
  \end{center}%
}